\documentclass[final,twoside,11pt]{entics} 
\usepackage{enticsmacro}
\usepackage{graphicx}
\usepackage{enumerate}
\usepackage[all]{xy}
\usepackage{multicol}
\usepackage{amsmath,amssymb,wasysym}
\usepackage{url}
\usepackage{cmll}
\usepackage{stmaryrd} %need for special interpretation bracket
\def\conf{MFPS 2023} 	%%Fill in the acronym for your conference (with year)
\def\myconf{mfps39}
\volume{3}			%Fill in the ENTICS volume number here
\def\volu{3}			% and here
\def\papno{13}		%Fill in your paper number here
\def\lastname{Lemay} %Lastnames appear in the running header 
\def\runtitle{Cartesian Differential Kleisli Categories} %Short title appears in the running header on even pages. 
\def\mydoi{12278}
\def\copynames{J.-S. P. Lemay}    %% Fill in the first initial and last name of the authors
\begin{document}
%%%Note the beginning and end of the frontmatter section that starts here%%%%%
\begin{frontmatter}
  \title{Cartesian Differential Kleisli Categories}% \thanksref{ALL}} 						%%Title here and the
% \thanks[ALL]{}   %%Text of \thanks[ALL} here..
 %%%%%%%%%%%%%%%%%%%%%%%%%%%%			This Thanks is optional.
  %%%%Now the author(s) names(s)%%%%%
    % \author{Jonathan Gallagher\thanksref{b}\thanksref{coemail}}		%last name and \thanksref{...} 
  \author{Jean-Simon Pacaud Lemay\thanksref{a}\thanksref{myemail}}	%%Note NO SPACE between 
    %%%Next come the addresses%%%%

   %%%Note: if both authors share same institution, only list the address once, after the second 
   %%%author. 
   %%%There also is a link from the first author to the co-author's address to show how to list 
   %%%affiliations to more than one institution, when needed. 
 % \address[b]{Center for Secure and Resilient Systems\\ HRL Laboratories\\ Malibu, California, USA} 
%  \thanks[coemail]{Email:  \href{mailto:jonathan@infinitylab.io} {\texttt{\normalshape
%        jonathan@infinitylab.io}}}
          \address[a]{School of Mathematical and Physical Sciences
 \\ Macquarie University\\				%or between \thanksrefs...
    Sydney, Australia}  							
   \thanks[myemail]{Email: \href{mailto:js.lemay@mq.edu.au} {\texttt{\normalshape
        js.lemay@mq.edu.au}}} 
\begin{abstract} Cartesian differential categories come equipped with a differential combinator which axiomatizes the fundamental properties of the total derivative from differential calculus. The objective of this paper is to understand when the Kleisli category of a monad is a Cartesian differential category. We introduce Cartesian differential monads, which are monads whose Kleisli category is a Cartesian differential category by way of lifting the differential combinator from the base category. Examples of Cartesian differential monads include tangent bundle monads and reader monads. We give a precise characterization of Cartesian differential categories which are Kleisli categories of Cartesian differential monads using abstract Kleisli categories. We also show that the Eilenberg-Moore category of a Cartesian differential monad is a tangent category. 
\end{abstract}
\end{frontmatter}
\section{Introduction}

Cartesian differential categories, introduced by Blute, Cockett, and Seely in \cite{blute2009cartesian}, provide the categorical foundations of multivariable differential calculus. Briefly, a Cartesian differential category (Def \ref{cartdiffdef}) is a category with finite products that comes equipped with a differential combinator $\mathsf{D}$, which for every map $f: A \to B$ produces its derivative $\mathsf{D}[f]: A \times A \to B$, satisfying seven axioms that are analogues of the fundamental properties of the total derivative, including the chain rule. Cartesian closed differential categories \cite{Cockett-2019} provide the categorical semantics of Ehrhard and Regnier's differential $\lambda$-calculus \cite{ehrhard2003differential}. Cartesian (closed) differential categories have found numerous applications in computer science such as for causal computations \cite{sprunger2019differentiable}, differentiable programming languages \cite{abadi2019simple,cruttwell2020categorical}, incremental computation \cite{alvarez2020cartesian,alvarez2019change}, game theory \cite{laird2013constructing}, and for automatic differentiation and machine learning \cite{cockettetal:LIPIcs:2020:11661,cruttwell2021categorical,wilson2021reverse}.

There is no shortage of interesting examples of Cartesian (closed) differential categories in the literature. In \cite{bauer2018directional}, Bauer et al.\! provide an Abelian functor calculus model of a Cartesian differential category by considering the homotopy category of the Kleisli category of the chain complex monad. The development of this example was quite important since it steered the theory of Cartesian differential categories into a new field. However, from a certain point of view, the construction of this example is somewhat unexpected since usually Cartesian differential categories are constructed as coKleisli categories. Indeed, important examples of Cartesian differential categories are the coKleisli categories of the comonads of categorical models of Differential Linear Logic \cite{blute2009cartesian,blute2015cartesian}, or more generally the coKleisli categories of Cartesian differential comonads \cite{ikonicoff2021cartesian}. In this case, the differential combinator in the coKleisli category is constructed using a natural transformation associated with the comonad. 

As such, a natural question to ask is whether it is possible to generate examples of Cartesian differential categories instead as Kleisli categories. The answer is not as straightforward as simply dualizing the construction from the coKleisli story, since the definition of a Cartesian differential category is not self-dual. Furthermore, while some of the necessary underlying structure, specifically the Cartesian left $k$-structure (Def \ref{LACdef}), comes for free in the coKleisli setting, they instead require extra assumptions for the Kleisli setting. Therefore, the Kleisli story will instead start with a monad on a Cartesian differential category and \emph{lift} the differential combinator to its Kleisli category. 

Lifting structure to Kleisli categories is always an interesting question in category theory, but it is also an important concept in computer science \cite{bucalo2003equational,mulry1994lifting,mulry2002lifting}. Following Moggi's groundbreaking work on the importance of monads for programming languages \cite{moggi1991notions}, maps in Kleisli categories are interpreted as effectful programs. Solutions to the lifting problem allow us to extend desirable structures or properties of the base programming language to the effectful programming language, which is always advantageous. By extending the differential combinator to effectful programs, we would be able to apply differential calculus-based techniques and algorithms on effectful programs. 

 In this paper, we introduce Cartesian differential monads (Def \ref{def:CDM}), which are precisely the kinds of monads on Cartesian differential categories which lift the Cartesian differential structure to their Kleisli category. Concretely, if $\mathbb{S}=(\mathbb{S}, \mu, \eta)$ is a Cartesian differential monad on a Cartesian differential category $\mathbb{X}$, then its Kleisli category $\mathsf{KL}(\mathbb{S})$ is a Cartesian differential category such that the canonical adjoint functors $\mathsf{L}_\mathbb{S}: \mathbb{X} \to \mathsf{KL}(\mathbb{S})$ and $\mathsf{R}_\mathbb{S}: \mathsf{KL}(\mathbb{S}) \to \mathbb{X}$ preserve the Cartesian differential structure (Thm \ref{thm:CDMKleisli}). In particular, this implies that the differential combinator in the Kleisli category $\mathsf{KL}(\mathbb{S})$ is the same as the one in the base category $\mathbb{X}$. Type-wise this makes sense since for a map of type $A \to \mathsf{S}(B)$, which is a map in the Kleisli category, its derivative is of type $A \times A \to \mathsf{S}(B)$, which is still a map in the Kleisli category. However, one must be careful since composition in $\mathsf{KL}(\mathbb{S})$ is different from the one in $\mathbb{X}$. In light of this, for a Cartesian differential monad, we ask that both the multiplication $\mu$ and unit $\eta$ be \emph{differential linear} maps. This allows the differential combinator to behave well with Kleisli composition, so that the differential combinator axioms, in particular the chain rule, hold in the Kleisli category. 

We consider two important examples of Cartesian differential monads. The first is the tangent bundle monad (Ex \ref{ex:tanbun}), which is a canonical monad for any Cartesian differential category which is built using the differential combinator. The maps in the Kleisli category of the tangent bundle monad can be interpreted as generalized vector fields (Ex \ref{ex:tanbunKleisli}), which were studied in \cite{alvarez2020cartesian,alvarez2019change}. The second kinds are the reader monads in a Cartesian \emph{closed} differential category (Ex \ref{ex:reader}). The Kleisli category for a reader monad encodes differentiation in context (Ex \ref{ex:readerKleisli}). Furthermore, using F\"{u}hrmann's notion of abstract Kleisli categories \cite{fuhrmann1999direct}, we can provide a precise characterization of which Cartesian differential categories are Kleisli categories of Cartesian differential monads (Prop \ref{propab1} \& \ref{propab2}). 

We also consider a slightly more general concept by introducing Kleisli differential combinators (Def \ref{def:KDC}), which provide a direct description of a differential combinator in a Kleisli category. On the one hand, for a monad with a Kleisli differential combinator, its Kleisli category is a Cartesian differential category. On the other hand, every Cartesian differential monad has a canonical Kleisli differential combinator. One slight advantage of Kleisli differential combinators compared to Cartesian differential monads is that the former does not require the base category to be a Cartesian differential category. That said, Cartesian differential monads are much simpler in comparison. However, we will explain how from a monad with a Kleisli differential combinator, we can always construct a Cartesian differential monad whose Kleisli category is precisely the same as the starting monad's. 

Lastly, another important question for a monad is what are its algebras and Eilenberg-Moore category. For a Cartesian differential monad, the answer is that its Eilenberg-Moore category is a tangent category (Thm \ref{thm:EMCDM}). While Cartesian differential categories formalize differential calculus over Euclidean spaces, tangent categories, introduced by Cockett and Cruttwell in \cite{cockett2014differential}, instead formalize differential calculus over smooth manifolds and their tangent bundles. As such, algebras of a Cartesian differential monad can be interpreted as abstract smooth manifolds. A differential object in a tangent category is an analogue of a Euclidean space, and the subcategory of differential objects is a Cartesian differential category. We show that for a Cartesian differential monad, an algebra is a differential object if and only if its algebra structure is differential linear. Furthermore, the Kleisli category embeds as a Cartesian differential category into the subcategory of differential objects of the Eilenberg-Moore category (Lem \ref{lem:diffobjCDM}).   

\section{Cartesian Differential Categories}

In this background section, if only to introduce notation and terminology, we quickly review Cartesian differential categories. In this paper, we will work with Cartesian differential categories relative to a fixed commutative semiring $k$, as introduced by Garner and the author in \cite{garner2020cartesian}. When $k= \mathbb{N}$, the semiring of natural numbers, we obtain precisely Blute, Cockett, and Seely's original definition of a Cartesian differential category in \cite{blute2009cartesian}. For a more in-depth introduction to Cartesian differential categories, we refer the reader to those papers: \cite{blute2009cartesian,garner2020cartesian}. 

The underlying structure of a Cartesian differential category is that of a Cartesian left $k$-linear category, which can be described as a category with finite products which is \emph{skew}-enriched over the category of $k$-modules and $k$-linear maps between them \cite{garner2020cartesian}. Essentially, this means that each hom-set is a $k$-module, so in particular, we have zero maps and can take the sum of maps, but also allow for maps which do not preserve zeroes or sums. Maps which do preserve the module structure are called $k$-linear maps. Cartesian left $\mathbb{N}$-linear categories are the same thing as Cartesian left additive categories \cite[Def 1.2.1]{blute2009cartesian}, and the $\mathbb{N}$-linear maps are precisely the additive maps \cite[Def 1.1.1]{blute2009cartesian}.

In an arbitrary category $\mathbb{X}$, we denote hom-sets by $\mathbb{X}(A,B)$, identity maps as $1_A: A \to A$, and we use the classical notation for composition, $\circ$, as opposed to diagrammatic order which was used in other papers on Cartesian differential categories, such as in \cite{blute2009cartesian}. For a category with finite products, we denote the product by $\times$, the projections as ${\pi_j: A_1 \times \hdots \times A_n \to A_j}$, the pairing operation as $\langle -, \hdots, - \rangle$, and we denote the terminal object as $\ast$, with the unique map to the terminal object as $\mathsf{t}_A: A \to \ast$.

\begin{definition} \label{LACdef} \cite[Sec 2.1]{garner2020cartesian} A \textbf{left $k$-linear category} is a category $\mathbb{X}$ such that each hom-set $\mathbb{X}(A,B)$ is a $k$-module with:
\begin{enumerate}[{\em (i)}]
\item Scalar multiplication $\cdot : k \times  \mathbb{X}(A,B) \to  \mathbb{X}(A,B)$;
\item Addition ${+: \mathbb{X}(A,B) \times \mathbb{X}(A,B) \to \mathbb{X}(A,B)}$;
\item Zero $0 \in \mathbb{X}(A,B)$;
\end{enumerate}
and such that pre-composition preserves the $k$-linear structure, that is, for suitable maps $f: A\to B$, $g: A \to B$, and $x: A^\prime \to A$, and all $r,s \in R$, the following equality holds: 
\begin{align}
   (r \cdot f + s \cdot g) \circ x = r \cdot (f \circ x) + s \cdot (g \circ x) 
\end{align}
A map $f: A\to B$ is said to be \textbf{$k$-linear} if post-composition by $f$ preserves the $k$-linear structure, that is, for all suitable maps $x: A^\prime \to A$ and $y: A^\prime \to A$, and for all $r,s \in R$, the following equality holds: 
\begin{align}
    f \circ (r \cdot x + s \cdot y) =   r \cdot (f \circ x) + s \cdot (f \circ y)
\end{align} 
A \textbf{Cartesian left $k$-linear category} is a left $k$-linear category $\mathbb{X}$ such that $\mathbb{X}$ has finite products\footnote{Note that in a Cartesian left $k$-linear category, the unique map to the terminal object is the zero map $0: A \to \ast$.} and all projection maps $\pi_j$ are $k$-linear maps. 
\end{definition}

Cartesian differential categories are Cartesian $k$-linear categories which come equipped with a differential combinator, which is an operator which sends maps to their derivative. The axioms of a differential combinator are analogues of the basic properties of the total derivative from multivariable differential calculus. There are various equivalent ways of stating the axioms of a Cartesian differential category. For this paper, we have chosen the one found in \cite[Def 2.6]{lemay2018tangent}, adapted for the $k$-linear setting. We also recall the definition of differential linear maps \cite[Def 2.7]{garner2020cartesian}, which are an important kind of map in a Cartesian differential category, and play a key role in the definition of Cartesian differential monads.

\begin{definition}\label{cartdiffdef} \cite[Sec 2.2]{garner2020cartesian} A \textbf{Cartesian $k$-differential category} is a Cartesian left $k$-linear category $\mathbb{X}$ equipped with a \textbf{differential combinator} $\mathsf{D}$, which is a family of functions $\mathsf{D}: \mathbb{X}(A,B) \to \mathbb{X}(A \times A,B)$, where for a $f: A \to B$, the map $\mathsf{D}[f]: A \times A \to B$ is called the \textbf{derivative} of $f$, and such that the following seven axioms hold:   
\begin{enumerate}[{\bf [CD.1]}]
\item \label{CDCax1} $\mathsf{D}[r\cdot f + s \cdot g] = r \cdot \mathsf{D}[f] + s \cdot \mathsf{D}[g]$
\item \label{CDCax2} $\mathsf{D}[f] \circ \langle \pi_1, r\cdot \pi_2 + s\cdot \pi_3 \rangle= r\cdot \left( \mathsf{D}[f] \circ \langle \pi_1, \pi_2 \rangle \right) + s \cdot \left( \mathsf{D}[f] \circ \langle \pi_1, \pi_2 \rangle \right)$ 
\item \label{CDCax3} $\mathsf{D}[1_A]=\pi_2$ and $\mathsf{D}[\pi_j] = \pi_{n+j}$ 
\item \label{CDCax4} $\mathsf{D}[\left\langle f_1, \hdots, f_n \right \rangle] = \left \langle  \mathsf{D}[f_1], \hdots, \mathsf{D}[f_n] \right \rangle$
\item \label{CDCax5} $\mathsf{D}[g \circ f] = \mathsf{D}[g] \circ \langle f \circ \pi_1, \mathsf{D}[f] \rangle$
\item \label{CDCax6} $\mathsf{D}\left[\mathsf{D}[f] \right] \circ \left \langle \pi_1, 0, 0, \pi_2 \right \rangle=  \mathsf{D}[f]$
\item \label{CDCax7} $\mathsf{D}\left[\mathsf{D}[f] \right] \circ \left \langle \pi_1, \pi_2, \pi_3, \pi_4 \right \rangle= \mathsf{D}\left[\mathsf{D}[f] \right] \circ \left \langle \pi_1, \pi_3, \pi_2, \pi_4 \right \rangle$
\end{enumerate}
\noindent 
\noindent A map $f: X \to Y$ is \textbf{differential linear}, or simply \textbf{$\mathsf{D}$-linear}, if $\mathsf{D}[f] = f \circ \pi_2$. 
\end{definition}

There is a sound and complete term logic for Cartesian differential categories \cite[Sec 4]{blute2009cartesian}, which is useful for intuition. So we write:
\[\mathsf{D}[f](a,b) := \frac{\mathsf{d}f(x)}{\mathsf{d}x}(a) \cdot b\] 
and the axioms of the differential combinator may be rewritten as follows: 
\begin{enumerate}[{\bf [CD.1]}]
\item $\frac{\mathsf{d}r \cdot f(x) + s \cdot g}{\mathsf{d}x}(a) \cdot b = r\cdot \frac{\mathsf{d}f(x)}{\mathsf{d}x}(a) \cdot b + s \cdot \frac{\mathsf{d}f(x)}{\mathsf{d}x}(a) \cdot b$
\item $\frac{\mathsf{d}f(x)}{\mathsf{d}x}(a) \cdot (r\cdot b + s \cdot c) = r\cdot \frac{\mathsf{d}f(x)}{\mathsf{d}x}(a) \cdot b + s \cdot \frac{\mathsf{d}f(x)}{\mathsf{d}x}(a) \cdot b$ 
\item  $\frac{\mathsf{d}x_i}{\mathsf{d}x_i}(a_1, \hdots, a_n) \cdot (b_1, \hdots, b_n) = b_i$ 
\item  $\frac{\mathsf{d}\left \langle f_1(x), \hdots, f_n \right \rangle}{\mathsf{d}x}(a) \cdot b = \left \langle \frac{\mathsf{d}f_1(x)}{\mathsf{d}x}(a) \cdot b, \hdots, \frac{\mathsf{d}f_n(x)}{\mathsf{d}x}(a) \cdot b \right \rangle$
\item $\frac{\mathsf{d}g\left(f(x) \right)}{\mathsf{d}x}(a) \cdot b = \frac{\mathsf{d}g(y)}{\mathsf{d}y}(f(a)) \cdot \left( \frac{\mathsf{d}f(x)}{\mathsf{d}x}(a) \cdot b \right)$
\item  $\frac{\mathsf{d}\frac{\mathsf{d}f(x)}{\mathsf{d}x}(y) \cdot z}{\mathsf{d}(y,z)}(a,0) \cdot (0,b) = \frac{\mathsf{d}f(x)}{\mathsf{d}x}(a) \cdot b$
\item $\frac{\mathsf{d}\frac{\mathsf{d}f(x)}{\mathsf{d}x}(y) \cdot z}{\mathsf{d}(y,z)}(a,b) \cdot (c,d) = \frac{\mathsf{d}\frac{\mathsf{d}f(x)}{\mathsf{d}x}(y) \cdot z}{\mathsf{d}(y,z)}(a,c) \cdot (b,d)$
\end{enumerate}
Briefly, the axioms of a differential combinator are that: \textbf{[CD.1]} the differential combinator is a $k$-linear morphism, \textbf{[CD.2]} derivatives are $k$-linear in their second argument, \textbf{[CD.3]} the derivative of identity maps and projections are projections, \textbf{[CD.4]} the derivative of a pairing is the pairing of the derivatives, \textbf{[CD.5]} the chain rule for the derivative of a composition, \textbf{[CD.6]} the derivative is differential linear in its second argument, and lastly \textbf{[CD.7]} is the symmetry of the partial derivatives. Furthermore, as explained in \cite[Lem 2.8]{lemay2018tangent}, it turns out that \textbf{[CD.4]} is in fact redundant and follows from \textbf{[CD.3]} and \textbf{[CD.5]}. For a differential linear map, its derivative is simply the starting map evaluated in the second argument:
\[\frac{\mathsf{d}f(x)}{\mathsf{d}x}(a) \cdot b = f(b)\] 
As it will be important below, it is worth mentioning now that every differential linear map is always $k$-linear \cite[Lem 2.2.2.(i)]{blute2009cartesian}. 

\begin{example} \label{ex:smooth} Arguably, the canonical example of a Cartesian differential category is the Lawvere theory of real smooth functions. So let $\mathbb{R}$ be the set of real numbers, and define $\mathsf{SMOOTH}$ to be the category whose objects are the Euclidean spaces $\mathbb{R}^n$ and whose maps are smooth functions between them. $\mathsf{SMOOTH}$ is a Cartesian $\mathbb{R}$-differential category where the differential combinator is defined as the total derivative of a smooth function. For a smooth function ${F: \mathbb{R}^n \to \mathbb{R}^m}$, which is in fact an $m$-tuple $F = \langle f_1, \hdots, f_m \rangle$ of smooth functions $f_i: \mathbb{R}^n \to \mathbb{R}$, the derivative ${\mathsf{D}[F]: \mathbb{R}^n \times \mathbb{R}^n \to \mathbb{R}^m}$ is defined as:
\[\mathsf{D}[F](\vec x, \vec y) := \left \langle \sum \limits^n_{i=1} \frac{\partial f_1}{\partial x_i}(\vec x) y_i, \hdots, \sum \limits^n_{i=1} \frac{\partial f_n}{\partial x_i}(\vec x) y_i \right \rangle\]
A smooth function $F: \mathbb{R}^n \to \mathbb{R}^m$ is $\mathsf{D}$-linear if and only if it is $\mathbb{R}$-linear in the classical sense. 
\end{example}

For a list of more examples of Cartesian differential categories, see \cite[Ex 2.5]{garner2020cartesian} 

\section{Cartesian Differential Monads}

In this section, we introduce Cartesian differential monads, which we will prove in the next section are monads which lift Cartesian differential structure to their Kleisli categories. If only to introduce notation, recall that a monad on a category $\mathbb{X}$ is a triple $\mathbb{S} := (\mathsf{S}, \mu, \eta)$ consisting of a functor ${\mathsf{S}: \mathbb{X} \to \mathbb{X}}$, and two natural transformations $\mu_A: \mathsf{S}\mathsf{S}(A) \to \mathsf{S}(A)$, called the multiplication, and ${\eta_A: A \to \mathsf{S}(A)}$, called the unit, and such that the following equalities hold: 
\begin{align}
    \mu_A \circ \eta_{\mathsf{S}(A)} = 1_{\mathsf{S}(A)} = \mu_A \circ \mathsf{S}(\eta_A) && \mu_A \circ \mu_{\mathsf{S}(A)} = \mu_A \circ \mathsf{S}(\mu_A)
\end{align}

To lift the Cartesian differential structure, we first need the ability to lift the Cartesian $k$-linear structure. Lifting Cartesian $k$-linear structure to coKleisli categories is automatic \cite[Prop 1.3.3]{blute2009cartesian}, in the sense that no extra assumption is needed on the comonad. The same is not true for Kleisli categories. Indeed, to lift products to the Kleisli category, one needs that the functor preserves products up to isomorphism, so $\mathsf{S}(A_1 \times \hdots \times A_n) \cong \mathsf{S}(A_1) \times \hdots \times \mathsf{S}(A_n)$ and $\mathsf{S}(\ast) \cong \ast$. To lift $k$-linear structure, one needs that the functor be a $k$-linear morphism between the hom-sets, and both the multiplication and the unit be $k$-linear maps. 

\begin{definition} Let $\mathbb{X}$ and $\mathbb{Y}$ be Cartesian left $k$-linear categories. A \textbf{strong Cartesian $k$-linear functor} is a functor $\mathsf{F}: \mathbb{X} \to \mathbb{Y}$ such that: 
\begin{enumerate}[{\em (i)}]
\item $\mathsf{F}$ preserves products up to isomorphism, that is, $\mathsf{t}_{\mathsf{F}(\ast)}: \mathsf{F}(\ast) \to \ast$ is an isomorphism, and the canonical natural transformations $\omega_{A_1, \hdots, A_n}: \mathsf{F}(A_1 \times \hdots \times A_n) \to \mathsf{F}(A_1) \times \hdots \times \mathsf{F}(A_n)$, defined as:
\begin{align}
    \omega_{A_1, \hdots, A_n} := \left \langle \mathsf{F}(\pi_1), \hdots, \mathsf{F}(\pi_n) \right \rangle
\end{align}
are natural isomorphisms. 
%So we have that . 
\item $\mathsf{F}$ is a $k$-linear morphism on hom-sets, that is, the following equality holds: 
\begin{align}
    \mathsf{F}(r \cdot f+ s \cdot g) = r \cdot \mathsf{F}(f) + s \cdot \mathsf{F}(g)
\end{align}
\end{enumerate}
A \textbf{strict Cartesian $k$-linear functor} is a strong Cartesian $k$-linear functor such that $\omega=1$, so in particular we have that $\mathsf{F}(\ast) = \ast$, $\mathsf{F}(A_1 \times \hdots \times A_n) = \mathsf{F}(A_1) \times \hdots \times \mathsf{F}(A_n)$, and $\mathsf{F}(\pi_j) = \pi_j$. 
\end{definition}

\begin{definition}\label{def:klinmonad} A \textbf{Cartesian $k$-linear monad} on a Cartesian left $k$-linear category $\mathbb{X}$ is a monad $\mathbb{S}$ such that the underlying functor $\mathsf{S}$ is a strong Cartesian $k$-linear functor, and  for all objects $A$, both $\eta_A$ and $\mu_A$ are $k$-linear maps. 
\end{definition}

It is also useful to note that Cartesian $k$-linear functors preserve $k$-linear maps: 

\begin{lemma}\label{lem:klinfunc}\cite[Lem 1.3.2]{blute2009cartesian} Let $\mathsf{F}: \mathbb{X} \to \mathbb{Y}$ be a strong Cartesian $k$-linear functor. Then $\mathsf{F}$ preserves $k$-linear maps, that is, if $f$ is a $k$-linear map in $\mathbb{X}$, then $\mathsf{F}(f)$ is a $k$-linear map in $\mathbb{Y}$. 
\end{lemma}

We may now properly define the notion of a Cartesian differential monad. Essentially, this is a monad whose monad structure is compatible with the differential combinator. So first we define Cartesian differential functors: 

\begin{definition}  Let $\mathbb{X}$ and $\mathbb{Y}$ be Cartesian $k$-differential categories. A \textbf{strong Cartesian $k$-differential functor} is a strong (resp. strict) Cartesian $k$-linear functor $\mathsf{F}: \mathbb{X} \to \mathbb{Y}$ such that for all maps ${f: A \to B}$, the following equality holds:
\begin{align} 
\mathsf{D}[\mathsf{F}(f)] = \mathsf{F}(\mathsf{D}[f]) \circ \omega^{-1}_{A,A}\label{SD}
\end{align}
Similarly, a \textbf{strict Cartesian $k$-differential functor} is a strict Cartesian $k$-linear functor $\mathsf{F}: \mathbb{X} \to \mathbb{Y}$ such that for all maps ${f: A \to B}$, $\mathsf{D}[\mathsf{F}(f)] = \mathsf{F}(\mathsf{D}[f])$. 
\end{definition}

We note that Cartesian differential functors preserve differential linear maps: 

\begin{lemma}\label{lem:dlinfunc} Let $\mathsf{F}: \mathbb{X} \to \mathbb{Y}$ be a strong Cartesian $k$-differential functor. If $f: A \to B$ is $\mathsf{D}$-linear in $\mathbb{X}$, then $\mathsf{S}(f): \mathsf{S}(A) \to \mathsf{S}(B)$ is also $\mathsf{D}$-linear in $\mathbb{Y}$. 
\end{lemma}
\begin{proof} First note that $\mathsf{F}(\pi_2) \circ \omega^{-1}_{A,A} = \pi_2$. So if $f$ is $\mathsf{D}$-linear, by using (\ref{SD}) we easily compute that: 
\begin{gather*}
\mathsf{D}[\mathsf{F}(f)] = \mathsf{F}(\mathsf{D}[f]) \circ \omega^{-1}_{A,A} = \mathsf{F}(f) \circ \mathsf{F}(\pi_2) \circ \omega^{-1}_{A,A} = \mathsf{F}(f) \circ \pi_2
\end{gather*}
    So we conclude that $\mathsf{F}(f)$ is also $\mathsf{D}$-linear. 
\end{proof}

We now define the main novel concept of this paper: 

\begin{definition}\label{def:CDM} A \textbf{Cartesian $k$-differential monad} on a Cartesian $k$-differential category $\mathbb{X}$ is a monad $\mathbb{S}$ such that the underlying functor $\mathsf{S}$ is a strong Cartesian $k$-differential functor, and for all objects $A$, both $\eta_A$ and $\mu_A$ are $\mathsf{D}$-linear maps.
\end{definition}

As differential linear maps are always $k$-linear maps, a Cartesian $k$-differential monad is also a Cartesian $k$-linear monad. From a higher category theory point of view, Cartesian differential monads are precisely monads in the (strict) 2-category \cite{street1972formal} of Cartesian differential categories, strong cartesian differential functors, and differential linear natural transformations (by which we mean a natural transformation that is differential linear at each component). Here are now some examples of Cartesian differential monads. 

\begin{example} Every Cartesian $k$-differential category always admits two trivial Cartesian $k$-differential monads: the identity monad and the constant monad that sends every object to the terminal object. 
\end{example}

\begin{example}\label{ex:tanbun} Every Cartesian differential category always has yet another Cartesian differential monad provided by its canonical tangent bundle monad. As the name suggests, this monad arises from the fact that every Cartesian differential category is a tangent category (which we review in Sec \ref{sec:tancats}), and that the tangent bundle functor of a tangent category always has a canonical monad structure \cite[Prop 3.4]{cockett2014differential}. In \cite{manzyuk2012tangent}, Manzyuk studied the tangent bundle monad of categorical models of the differential $\lambda$-calculus. In \cite{alvarez2020cartesian,lemay2021cartesian}, Alvarez-Picallo and the author studied the tangent bundle monad for Cartesian \emph{difference} categories, a slight generalization of Cartesian differential categories by adding extra nilpotent structure. So for a Cartesian $k$-differential category $\mathbb{X}$, its \textbf{tangent bundle monad} \cite[Sec 6.1]{lemay2021cartesian} is the monad $\mathbb{T} := (\mathsf{T}, \mu, \eta)$ which is defined as follows: 
\begin{gather*}
    \mathsf{T}(A) = A \times A \qquad \mathsf{T}(f) = \left \langle f \circ \pi_1, \mathsf{D}[f] \right \rangle \\
    \mu_A = \left \langle \pi_1, \pi_2 + \pi_3 \right \rangle \qquad \eta_A = \langle 1_A, 0 \rangle 
\end{gather*}
In term calculus notation, the tangent bundle on maps is given by:
\[\mathsf{T}(f)(a,b) = \left(f(a), \frac{\mathsf{d}f(x)}{\mathsf{d}x}(a) \cdot b \right)\] 
All the necessary requirements that need to be checked for the tangent bundle to be a Cartesian $k$-differential monad can be found in \cite[Lem 6.4]{lemay2021cartesian}. In particular, in the term calculus, (\ref{SD}) is expressed as the equality: 
\[\frac{\mathsf{d} \left( f(x), \frac{\mathsf{d}f(u)}{\mathsf{d}u}(x) \cdot y \right)}{\mathsf{d}(x,y)}(a,b) \cdot (c,d) = \left( \frac{\mathsf{d}f(x)}{\mathsf{d}x}(a) \cdot c, \frac{\mathsf{d} \frac{\mathsf{d}f(u)}{\mathsf{d}u}(x) \cdot y }{\mathsf{d}(x,y)}(a,c) \cdot (b,d) \right)\]
which the reader can check for themselves using the term calculus as a good exercise. 
\end{example}

\begin{example} It may also be useful to consider a concrete example of the tangent bundle monad. So consider the Cartesian $\mathbb{R}$-differential category $\mathsf{SMOOTH}$ from Ex \ref{ex:smooth}. Then the resulting tangent bundle monad $(\mathsf{T}, \mu, \eta)$ is given as follows (for all $\vec x, \vec y, \vec z, \vec w \in \mathbb{R}^n$): 
\begin{gather*}
    \mathsf{T}(\mathbb{R}^n) = \mathbb{R}^n \times \mathbb{R}^n \qquad \mathsf{T}(F)(\vec x, \vec y) = \left( F(\vec x), \mathsf{D}[F](\vec x, \vec y) \right) \\
    \mu_{\mathbb{R}^n}(\vec x, \vec y, \vec z, \vec w) = (\vec x, \vec y + \vec z) \qquad \eta_{\mathbb{R}^n}(\vec x) = ( \vec x, \vec 0 )
\end{gather*}
Let us work out what (\ref{SD}) states for a smooth function ${f: \mathbb{R} \to \mathbb{R}}$. On the one hand:
\[\mathsf{D}[f](x,y) = f^\prime(x)y\] 
where $f^\prime$ is the derivative of $f$, and so applying $\mathsf{T}$ we get:
\[\mathsf{T}(\mathsf{D}[f])(x,y,z,w) = (f^\prime(x) y, f^{\prime \prime}(x)yz + f^\prime(x) w)\]
On the other hand, first applying $\mathsf{T}$ we get: 
\[\mathsf{T}(f)(x,y) = \left( f(x), f^\prime(x)y \right)\] 
and so taking its derivative we then get: 
\[\mathsf{D}[\mathsf{T}(f)](x, y, z,w) = \left(f^\prime(x)z, f^{\prime\prime}(x)yz + f^\prime(x)w  \right)\] 
Thus, we see that:
\[\mathsf{D}[\mathsf{T}(f)](x, y, z,w) = \mathsf{T}(\mathsf{D}[f])(x,z,y,w)\] 
where note the swapping of the middle two arguments. 
\end{example}

\begin{example}\label{ex:reader} In a Cartesian \emph{closed} differential category, reader monads are Cartesian differential monads. Briefly, a Cartesian closed $k$-differential category \cite[Def 2.9]{garner2020cartesian} is a Cartesian $k$-differential category $\mathbb{X}$ which is also Cartesian closed such that differential combinator is compatible with the closed structure. In the term calculus, this is expressed as:
\[\frac{\mathsf{d}\lambda y.f(y,x)}{\mathsf{d}x}(a) \cdot b = \lambda y. \frac{\mathsf{d}f(y,x)}{\mathsf{d}x}(a) \cdot b\] 
which says that the derivative of a Curry of a map is the Curry of the partial derivative of said map. Every model of the differential $\lambda$-calculus \cite{ehrhard2003differential} induces a Cartesian closed differential category, and conversely, every Cartesian closed differential category gives rise to a model of the differential $\lambda$-calculus. Now let $[-,-]$ be the internal hom. Then every object $C$ induces a monad $\mathbb{R}(C) := \left( [C,-], \mu^C, \eta^C \right)$, called the reader monad, and is defined as follows (using the term calculus): 
\begin{gather*}
[C,-](A) := [C,A] \quad \quad \quad [C,-](f)\left( g(-) \right) = [C,f]\left( g(-) \right)  = \lambda x. f(g(x)) \\
\mu^C_A\left( F(-)(-) \right)  = \lambda x. F(x)(x) \quad \quad \quad \eta^C_A(a) = \lambda x. a  
\end{gather*}
It is well-known that $[C, A \times B] \cong [C,A] \times [C,B]$, and it is easy to check that, by  the left $k$-linear structure, we have that:
\[ [C, r\cdot f + s \cdot g] = r \cdot [C,f] + s \cdot [C, g]\] 
Furthermore, we also have that:
\[\frac{\mathsf{d}\lambda y.f(k(-))}{\mathsf{d}\left( k(-) \right) }\left( g(-) \right)  \cdot h(-)   = \lambda y. \frac{\mathsf{d}f(x)}{\mathsf{d}x}(g(y)) \cdot h(y)\] 
So we conclude that $\mathbb{R}(C)$ is a Cartesian $k$-differential monad. 
\end{example}

\section{Kleisli Categories of Cartesian Differential Monads}

In this section, we will now prove that the Kleisli category of a Cartesian differential monad is a Cartesian differential category. As we will be working with Kleisli categories, we will use notation inspired by \cite{blute2015cartesian} for Kleisli categories, and use interpretation brackets $\llbracket - \rrbracket$ to help distinguish between composition in the base category and Kleisli composition. So for a monad $\mathbb{S}$ on a category $\mathbb{X}$, let $\mathsf{KL}(\mathbb{S})$ denote its Kleisli category, which recall is the category whose objects are the same as $\mathbb{X}$ and where a map $f: A \to B$ in the Kleisli category is a map of type $\llbracket f \rrbracket: A \to \mathsf{S}(B)$ in the base category, that is, $\mathsf{KL}(\mathbb{S})(A,B) = \mathbb{X}(A, \mathsf{S}(B))$. The identity maps $\llbracket 1_A \rrbracket: A \to \mathsf{S}(A)$ in the Kleisli category is given by the monad unit, so: 
\begin{align}
    \llbracket 1_A \rrbracket  := \eta_A
\end{align}
Composition of Kleisli maps ${\llbracket f \rrbracket: A \to \mathsf{S}(B)}$ and $\llbracket g \rrbracket: B \to \mathsf{S}(C)$ is defined as:
\begin{align}
    \llbracket g \circ f \rrbracket := \mu_C \circ \mathsf{S}\left( \llbracket g \rrbracket \right) \circ \llbracket f \rrbracket
\end{align}
Let $\mathsf{L}_\mathbb{S}: \mathbb{X} \to \mathsf{KL}(\mathbb{S})$ be the canonical inclusion functor which is defined on objects as $\mathsf{L}_\mathbb{S}(A)=A$ and on maps ${f: A \to B}$ as: 
\begin{align}
    \llbracket \mathsf{L}_\mathbb{S}(f) \rrbracket := \eta_B \circ f
\end{align}
Also let $\mathsf{R}_\mathbb{S}: \mathsf{KL}(\mathbb{S}) \to \mathbb{X}$ be the right adjoint of $\mathsf{L}_\mathbb{S}$, which is defined on objects as $\mathsf{R}_\mathbb{S}(A) = \mathsf{S}(A)$ and on a Kleisli map $\llbracket f \rrbracket: A \to \mathsf{S}(B)$ as:
\begin{align}
    \mathsf{R}_\mathbb{S}(f) = \mu_B \circ \mathsf{S}\left( \llbracket f \rrbracket \right)
\end{align} 

We begin by explaining how for a Cartesian $k$-linear monad, its Kleisli category is a Cartesian left $k$-linear category, whose structure is lifted from the base category. 

\begin{lemma}\label{lem:KleisliCLAC} Let $\mathbb{S}$ be a Cartesian $k$-linear monad on a Cartesian left $k$-linear category $\mathbb{X}$. Then the Kleisli category $\mathsf{KL}(\mathbb{S})$ is a Cartesian left $k$-linear category where: 
\begin{enumerate}[{\em (i)}]
\item The terminal object and product on objects are defined as in $\mathbb{X}$
\item Projections are defined as $\llbracket  \pi_j \rrbracket = \eta_{A_j} \circ \pi_j$
\item The pairing is defined as $\llbracket \langle f_1, \hdots, f_n  \rangle \rrbracket = \omega^{-1}_{A_1,\hdots,A_n} \circ \left \langle \llbracket f_1 \rrbracket, \hdots, \llbracket f_n \rrbracket \right \rangle$;
\item The $k$-module structure on $\mathsf{KL}(\mathbb{S})(A,B)$ is the same as $\mathbb{X}(A,\mathsf{S}(A))$, that is, given as follows: 
\begin{enumerate}
\item The scalar multiplication by $\llbracket r \cdot f \rrbracket = r \cdot \llbracket f \rrbracket$;
\item The addition by $\llbracket  f+g \rrbracket = \llbracket f \rrbracket + \llbracket g \rrbracket$;
\item Zero maps by $\llbracket  0 \rrbracket = 0$.
\end{enumerate}
\end{enumerate}
    Furthermore, $\mathsf{L}_\mathbb{S}: \mathbb{X} \to \mathsf{KL}(\mathbb{S})$ is a strict Cartesian $k$-linear functor and $\mathsf{R}_\mathbb{S}: \mathsf{KL}(\mathbb{S}) \to \mathbb{X}$ is a strong Cartesian $k$-linear functor. 
\end{lemma}
\begin{proof} It is known (though seemingly folkloric), that if the functor of a monad preserves products up to isomorphism, then the products lift to the Kleisli category. Next, we must show that pre-composition in the Kleisli category preserves the $k$-linear structure. This follows from the fact that $\mathsf{S}$ preserve $k$-linearity and $\mu$ is a $k$-linear map:
\begin{gather*}
\llbracket (r \cdot f + s \cdot g) \circ x \rrbracket = \mu_C \circ \mathsf{S}\left(\llbracket (r \cdot f + s \cdot g) \rrbracket \right) \circ \llbracket x \rrbracket = \mu_C \circ \mathsf{S}\left( r \cdot \llbracket f \rrbracket + s \cdot \llbracket  g \rrbracket \right) \circ \llbracket x \rrbracket\\
= \mu_C \circ \left(r \cdot  \mathsf{S}\left( \llbracket f \rrbracket \right) + s \cdot  \mathsf{S}\left( \llbracket g \rrbracket \right) \right) \circ \llbracket x \rrbracket = \left( r\cdot \left( \mu_C \circ  \mathsf{S}\left( \llbracket f \rrbracket \right) \right) + s\cdot \left( \mu_C \circ  \mathsf{S}\left( \llbracket g \rrbracket \right) \right) \right) \circ \llbracket x \rrbracket \\= r\cdot \left( \mu_C \circ  \mathsf{S}\left( \llbracket f \rrbracket \right) \circ \llbracket x \rrbracket \right)  + s\cdot \left( \mu_C \circ  \mathsf{S}\left( \llbracket g \rrbracket \right) \circ \llbracket x \rrbracket \right) = r\cdot \llbracket f \circ x \rrbracket + s \cdot \llbracket g \circ x \rrbracket = \llbracket r \cdot (f \circ x) + s \cdot (g \circ x) \rrbracket 
\end{gather*}
So the Kleisli category is indeed a left $k$-linear category. Lastly, since the projections $\pi_j$ are $k$-linear in $\mathbb{X}$, and that $\llbracket \pi_j \rrbracket = \llbracket \mathsf{L}_\mathbb{S}(\pi_j) \rrbracket$, it is straightforward to check that $\llbracket \pi_j \rrbracket$ are also $k$-linear in $\mathsf{KL}(\mathbb{S})$. So we conclude that the Kleisli category is a Cartesian left $k$-linear category as desired. Since $\mathsf{S}$ preserves the $k$-linear structure, and both $\eta$ and $\mu$ are $k$-linear maps, it is clear to see that $\mathsf{L}_\mathbb{S}$ and $\mathsf{R}_\mathbb{S}$ are strict and strong Cartesian $k$-linear functors respectively. 
\end{proof}

We now state the main result of this paper that the Kleisli category of a Cartesian differential monad is a Cartesian differential category, by way of lifting the differential combinator to the Kleisli category. Furthermore, a map is differential linear in the Kleisli category if and only if it is differential linear in the base category. 

\begin{theorem}\label{thm:CDMKleisli} Let $\mathbb{S}$ be a Cartesian $k$-differential monad on a Cartesian $k$-differential category $\mathbb{X}$ with differential combinator $\mathsf{D}$. Then:
\begin{enumerate}[{\em (i)}]
\item The Kleisli category $\mathsf{KL}(\mathbb{S})$ is a Cartesian $k$-differential category where the Cartesian left $k$-linear structure is defined as in Lemma \ref{lem:KleisliCLAC} and the differential combinator $\mathsf{D}_\mathbb{S}$ is defined on a Kleisli map $\llbracket f \rrbracket: A \to \mathsf{S}(B)$ as follows:
\begin{align}
    \llbracket \mathsf{D}_\mathbb{S}[f] \rrbracket := \mathsf{D}\left[ \llbracket f \rrbracket \right]: A \times A \to \mathsf{S}(B)
\end{align}
\item A map $f$ in $\mathsf{KL}(\mathbb{S})$ is $\mathsf{D}_\mathbb{S}$-linear if and only if $\llbracket f \rrbracket$ is $\mathsf{D}$-linear in $\mathbb{X}$. 
\item ${\mathsf{L}_\mathbb{S}: \mathbb{X} \to \mathsf{KL}(\mathbb{S})}$ is a strict Cartesian $k$-differential functor and $\mathsf{R}_\mathbb{S}: \mathsf{KL}(\mathbb{S}) \to \mathbb{X}$ is a strong Cartesian $k$-differential functor. 
\end{enumerate}
\end{theorem}
\begin{proof} \textbf{[CD.1]} is straightforward to check. To prove \textbf{[CD.2]}, \textbf{[CD.6]}, and \textbf{[CD.7]}, we use the fact that $\mathsf{L}_\mathbb{S}$ preserves the Cartesian left $k$-linear structure strictly and that $\llbracket g \circ \mathsf{F}_{\mathsf{S}}(f) \rrbracket = \llbracket g \rrbracket \circ f$. So for \textbf{[CD.6]} we compute: 
\begin{gather*}
    \llbracket \mathsf{D}_\mathbb{S}\left[ \mathsf{D}_\mathbb{S}[f] \right] \circ \langle \pi_1, 0, 0, \pi_2 \rangle \rrbracket =    \llbracket \mathsf{D}_\mathbb{S}\left[ \mathsf{D}_\mathbb{S}[f] \right] \circ \mathsf{L}_\mathbb{S}\left( \langle \pi_1, 0, 0, \pi_2 \rangle \right) \rrbracket = \llbracket \mathsf{D}_\mathbb{S}\left[ \mathsf{D}_\mathbb{S}[f] \right] \rrbracket \circ  \langle \pi_1, 0, 0, \pi_2 \rangle \\
    = \mathsf{D}\left[ \llbracket \mathsf{D}_\mathbb{S}[f]  \rrbracket \right] \circ \langle \pi_1, 0, 0, \pi_2 \rangle = \mathsf{D}\left[ \mathsf{D}\left[ \llbracket f \rrbracket \right] \right] \circ \langle \pi_1, 0, 0, \pi_2 \rangle = \mathsf{D}\left[ \llbracket f \rrbracket \right] = \llbracket \mathsf{D}_\mathbb{S}[f] \rrbracket
\end{gather*}
\textbf{[CD.2]} and \textbf{[CD.7]} are shown in a similar fashion. For \textbf{[CD.3]} and \textbf{[CD.5]}, we use the fact that if $g$ is $\mathsf{D}$-linear, then $\mathsf{D}[g \circ f] = g \circ \mathsf{D}[f]$ \cite[Lem 2.2.2.(vi)]{blute2009cartesian}. Then \textbf{[CD.3]} follows from the fact that $\eta$ is $\mathsf{D}$-linear:
\begin{gather*}
    \llbracket \mathsf{D}_\mathbb{S}[1_A] \rrbracket = \mathsf{D}\left[ \llbracket 1_A \rrbracket \right] = \mathsf{D}[\eta_A] = \eta_A \circ \pi_2 = \llbracket \pi_2 \rrbracket \\
    \llbracket \mathsf{D}_\mathbb{S}[\pi_j] \rrbracket = \mathsf{D}\left[ \llbracket \pi_j \rrbracket \right] = \mathsf{D}[\eta_{A_j} \circ \pi_j] = \eta_{A_j} \circ \mathsf{D}[\pi_j] = \eta_A \circ \pi_{n+j} = \llbracket \pi_{n+j} \rrbracket
\end{gather*}
For \textbf{[CD.5]} we use that $\mu$ is $\mathsf{D}$-linear and (\ref{SD}):
\begin{gather*}
    \llbracket \mathsf{D}_\mathbb{S}[g \circ f] \rrbracket = \mathsf{D}\left[ \llbracket g\circ f \rrbracket \right] = \mathsf{D}\left[\mu_C \circ \mathsf{S}\left( \llbracket g \rrbracket \right) \circ \llbracket f \rrbracket \right] = \mu_C \circ \mathsf{D}\left[  \mathsf{S}\left( \llbracket g \rrbracket \right) \circ \llbracket f \rrbracket \right]  \\
    = \mu_C \circ \mathsf{D}\left[\mathsf{S}\left( \llbracket g \rrbracket \right) \right] \circ \left \langle \llbracket f \rrbracket \circ \pi_1,  \mathsf{D}\left[ \llbracket f \rrbracket \right] \right \rangle = \mu_C \circ \mathsf{S}\left( \mathsf{D}\left[\llbracket g \rrbracket \right] \right) \circ \omega^{-1}_{\mathsf{S}(B), \mathsf{S}(B)} \circ \left \langle \llbracket f \circ \mathsf{L}_\mathbb{S}(\pi_1) \rrbracket ,  \mathsf{D}\left[ \llbracket f \rrbracket \right] \right \rangle \\
    = \mu_C \circ \mathsf{S}\left(  \llbracket \mathsf{D}_\mathbb{S}[g] \rrbracket \right) \circ \omega^{-1}_{\mathsf{S}(B), \mathsf{S}(B)} \circ \left \langle \llbracket f \circ \pi_1 \rrbracket ,  \llbracket \mathsf{D}_\mathbb{S}[f] \rrbracket \right \rangle =  \mu_C \circ \mathsf{S}\left(  \llbracket \mathsf{D}_\mathbb{S}[g] \rrbracket \right) \circ  \left \llbracket \left \langle f \circ \pi_1 ,  \mathsf{D}_\mathbb{S}[f]  \right \rangle \right \rrbracket \\
    = \left \llbracket \mathsf{D}_\mathbb{S}[g] \circ \left \langle f \circ \pi_1 ,  \mathsf{D}_\mathbb{S}[f]  \right \rangle \right \rrbracket 
\end{gather*}
Lastly, recall that we mentioned that \textbf{[CD.4]} follows from \textbf{[CD.3]} and \textbf{[CD.5]} \cite[Lem 2.8]{lemay2018tangent}. So we conclude that the Kleisli category $\mathsf{KL}(\mathbb{S})$ is indeed a Cartesian $k$-differential category as desired. Now observe that $\llbracket f \circ \pi_2 \rrbracket = \llbracket f \rrbracket \circ \pi_2$. As such, it is clear that $\llbracket f \rrbracket$ is $\mathsf{D}_\mathbb{S}$-linear if and only if $\llbracket f \rrbracket$ is $\mathsf{D}$-linear. Furthermore, since $\eta$ is $\mathsf{D}$-linear, we compute that: 
\begin{gather*}
    \llbracket \mathsf{L}_\mathbb{S}\left( \mathsf{D}[f] \right) \rrbracket= \eta_B \circ \mathsf{D}[f] = \mathsf{D}[\eta_B \circ f] = \mathsf{D}\left[ \llbracket \mathsf{L}_\mathbb{S}(f) \rrbracket \right] = \llbracket  \mathsf{D}_\mathbb{S}\left[\mathsf{L}_\mathbb{S}(f) \right] \rrbracket
\end{gather*}
So $\mathsf{L}_\mathbb{S}$ is a strict Cartesian $k$-differential functor. Thus by Lem \ref{lem:dlinfunc}, if $f$ is $\mathsf{D}$-linear, then $\mathsf{L}_\mathbb{S}(f)$ is $\mathsf{D}_\mathbb{S}$-linear. Lastly, using that $\mu$ is $\mathsf{D}$-linear and \cite[Lem 2.2.2.(vi)]{blute2009cartesian} again, it follows from (\ref{SD}) that: 
\begin{gather*}
    \mathsf{R}_\mathbb{S}\left( \mathsf{D}_\mathbb{S}[f] \right) \circ \omega^{-1}_{A,A}  = \mu_B \circ \mathsf{S}\left( \llbracket \mathsf{D}_\mathbb{S}[f] \rrbracket \right) \circ \omega^{-1}_{A,A} = \mu_B \circ \mathsf{S}\left(  \mathsf{D} \left[ \llbracket f  \rrbracket \right] \right) \circ \omega^{-1}_{A,A}\\
    = \mu_B \circ \mathsf{D}\left[ \mathsf{S}\left( \llbracket f  \rrbracket \right) \right] =  \mathsf{D}\left[\mu_B \circ \mathsf{S}\left( \llbracket f  \rrbracket \right) \right] =  \mathsf{D}\left[ \mathsf{R}_\mathbb{S}(f) \right] 
\end{gather*}
So $\mathsf{R}_\mathbb{S}$ is a strong Cartesian $k$-differential functor as desired. 
\end{proof}

\begin{example} Trivially, for the identity monad, its Kleisli category is precisely the starting Cartesian differential category, while for the constant monad, its Kleisli category is equivalent to the terminal category (i.e. the category with one object and only the identity map). 
\end{example}

\begin{example} \label{ex:tanbunKleisli} As explained in \cite[Sec 6.2]{alvarez2020cartesian}, for the tangent bundle monad, maps in its Kleisli category are interpreted as generalized vector fields. Indeed, if $\mathbb{X}$ is a Cartesian $k$-differential category, then a map in $\mathsf{KL}(\mathbb{T})$ is actually a pair of maps $\llbracket f \rrbracket = \langle f_1, f_2 \rangle: A \to B \times B$. In $\mathsf{KL}(\mathbb{T})$, the identity map is $\llbracket 1_A \rrbracket = \langle 1_A, 0 \rangle$, while composition is given by:
\[\llbracket g \circ f \rrbracket = \left \langle g_1 \circ f_1, g_2 \circ f_1 + \mathsf{D}[g_1] \circ \langle f_1, f_2 \rangle + \mathsf{D}[g_2] \circ \langle f_1, f_2 \rangle \right \rangle\] 
and the differential combinator is given by:
\[\llbracket \mathsf{D}_\mathbb{T}[f] \rrbracket := \left \langle \mathsf{D}[f_1], \mathsf{D}[f_2] \right\rangle\] 
The fact that, even with this much more complicated composition, the differential combinator still satisfies the chain rule is quite remarkable. Among these maps are the vector fields, in the tangent category sense \cite[Def 3.1]{cockett2014differential}, which are maps of the form $\llbracket v \rrbracket = \langle 1_A, v_2 \rangle: A \to A \times A$. Note that the derivative of a vector field is no longer a vector field. For more details on the Kleisli category of the tangent bundle monad, see \cite[Sec 6.3]{lemay2021cartesian}. 
\end{example}

\begin{example}For the tangent bundle monad on $\mathsf{SMOOTH}$, maps in the Kleisi category are smooth functions $\llbracket F \rrbracket: \mathbb{R}^n \to \mathbb{R}^n \times \mathbb{R}^n$ where $\llbracket F \rrbracket(\vec x) = (f_1(\vec x) , f_2(\vec x))$ for smooth functions $f_1: \mathbb{R}^n \to \mathbb{R}^n$ and $f_2: \mathbb{R}^n \to \mathbb{R}^n$. A vector field is of the form $\llbracket V \rrbracket(\vec x) = (\vec x, v_2(\vec x))$, and these correspond precisely to the vector fields of $\mathbb{R}^n$ in the classical differential geometry. Let us work out the composition of two vector fields $\llbracket V \rrbracket: \mathbb{R} \to \mathbb{R}$ and $\llbracket W \rrbracket: \mathbb{R} \to \mathbb{R}$, which is:
\[\llbracket W \circ V \rrbracket(x) = \left( x, w_2(x) + v_2(x) + w^\prime_2(x)v_2(x) \right)\]  
\end{example}

\begin{example} \label{ex:readerKleisli} For reader monads, their Kleisli categories capture partial differentiation. First observe that if $\mathbb{X}$ is a Cartesian closed $k$-differential category, then maps in $\mathsf{KL}(\mathbb{R}^C)$, which are of the form ${A \to [C,A]}$, correspond precisely to maps of the form ${C \times A \to B}$ via Currying and unCurrying. As such, the Kleisli category of the reader monad of $C$ is isomorphic to the coKleisli category of the comonad $C \times -$, which is better known as the simple slice category over $C$, denoted $\mathbb{X}[C]$ \cite[Sec 2.1]{blute2015cartesian}. Concretely, $\mathbb{X}[C]$ is the category whose objects are the same as $\mathbb{X}$ but where $\mathbb{X}[C](A,B) = \mathbb{X}(C \times A, B)$. As explained in \cite[Cor 4.5.2]{blute2009cartesian}, the simple slice category $\mathbb{X}[C]$ is again a Cartesian $k$-differential category but whose differential combinator $\mathsf{D}^C$ is given by partial differentiation, that is, for a map $f: C \times A \to B$, $\mathsf{D}^C[f]: C \times A \times A \to B$ is its partial derivative in only the variable $A$. In the term calculus:

\[\mathsf{D}^C[f](c,a,b) = \frac{\mathsf{d}f(c,x)}{\mathsf{d}x}(a) \cdot b := \frac{\mathsf{d}f(z,x)}{\mathsf{d}(z,x)}(c,a) \cdot (0,b) \] 
However, the partial derivative of a map $f: C \times A \to B$ is the same as taking the total derivative of its Curry $\lambda(f): A \to [C,B]$ and then unCurrying, that is:
\[\mathsf{D}^C[f] = \lambda^{-1}\left( \mathsf{D}[\lambda(f)] \right)\] 
As such, the lifting of the differential combinator $\mathsf{D}$ to $\mathsf{KL}(\mathbb{R}^C)$ does indeed recapture the fact that we can define partial differentiation in context $C$. For more details on partial differentiation in Cartesian differential categories, see \cite{blute2009cartesian,blute2015cartesian}. 
\end{example}

\section{Cartesian Differential Abstract Kleisli Categories}

The goal of this section is to give a precise characterization of the Cartesian differential categories which are the Kleisli categories of Cartesian differential monads. To achieve this, we use abstract Kleisli categories (sometimes also called thunk-force categories), which were introduced by F\"{u}hrmann in \cite{fuhrmann1999direct}. Abstract Kleisli categories provide a direct description of Kleisli categories of monads. For any monad, its Kleisli category is an abstract Kleisli category, and conversely, every abstract Kleisli category is isomorphic to the Kleisli category of a monad. In this section, we introduce the notion of Cartesian differential abstract Kleisli categories, and prove how these correspond to Kleisli categories of Cartesian differential monads. 

%So let us begin by reviewing abstract Kleisli categories  \cite[Def 2.1]{fuhrmann1999direct}.

\begin{definition}\cite[Def 2.1]{fuhrmann1999direct} An \textbf{abstract Kleisli structure} on a category $\mathbb{X}$ is a triple $(\mathsf{S}, \epsilon, \vartheta)$ consisting of an endofunctor ${\mathsf{S}: \mathbb{X} \to \mathbb{X}}$, a natural transformation $\epsilon_A: \mathsf{S}(A) \to A$, and a family of maps $\vartheta_A: A \to \mathsf{S}(A)$ (which are not necessarily natural), such that $\vartheta_{\mathsf{S}(A)}: \mathsf{S}(A) \to \mathsf{S}\mathsf{S}(A)$ is a natural transformation and also that the following equalities hold:
\begin{align}
    \epsilon_A \circ \vartheta_A = 1_A && \epsilon_{\mathsf{S}(A)} \circ \mathsf{S}(\vartheta_A) = 1_{\mathsf{S}(A)} && \vartheta_{\mathsf{S}(A)} \circ \vartheta_A = \mathsf{S}(\vartheta_A) \circ \vartheta_A
\end{align}
An \textbf{abstract Kleisli category} is a category $\mathbb{X}$ equipped with an abstract Kleisli structure $(\mathsf{S}, \epsilon, \vartheta)$. A map ${f: A \to B}$ is said to be \textbf{$\vartheta$-natural} \cite[Def 2.3]{fuhrmann1999direct} (also sometimes called a thunkable map) if the following equality holds: 
\begin{align}
    \vartheta_B \circ f = \mathsf{S}(f) \circ \vartheta_A
\end{align}
We denote the subcategory of $\vartheta$-natural maps of $\mathbb{X}$ by $\vartheta\text{-}\mathsf{nat}[\mathbb{X}]$.
\end{definition}

The subcategory of $\vartheta$-natural maps has a canonical monad whose Kleisli category is isomorphic to the starting abstract Kleisli category. 

\begin{lemma} \label{lem:ep-com} \cite[Sec 2]{fuhrmann1999direct} Let $\mathbb{X}$ be an abstract Kleisli category with abstract Kleisli structure $(\mathsf{S}, \epsilon, \vartheta)$. Define the natural transformation ${\delta_A: \mathsf{S}(A) \to \mathsf{S}\mathsf{S}(A)}$ as $\delta_A= \mathsf{S}(\epsilon_A)$. Then $\mathbb{S} = (\mathsf{S}, \delta, \vartheta)$ is a monad on $\vartheta\text{-}\mathsf{nat}[\mathbb{X}]$ such that the functor $\mathsf{G}_\vartheta: \mathbb{X} \to \mathsf{KL}(\mathbb{S})$, defined on objects as $\mathsf{G}_\vartheta(A)=A$ and on a map $f: A \to B$ as:
\begin{align}
  \llbracket \mathsf{G}_\mathsf{S}(f) \rrbracket = \mathsf{S}(f) \circ \vartheta_A
\end{align}
is an isomorphism with inverse $\mathsf{G}^{-1}_\vartheta: \mathsf{KL}(\mathbb{S}) \to \mathbb{X}$, defined on objects as $\mathsf{G}_\vartheta(A)=A$ and on a Kleisli map ${\llbracket f \rrbracket: A \to \mathsf{S}(A)}$ as:
\begin{align}
    \mathsf{G}^{-1}_\vartheta\left( \llbracket f \rrbracket \right) = \epsilon_B \circ \llbracket f \rrbracket
\end{align}
\end{lemma}

We now wish to equip abstract Kleisli categories with Cartesian differential structure. 

\begin{definition} A \textbf{Cartesian $k$-differential abstract Kleisli category} is an abstract Kleisli category $\mathbb{X}$ with abstract Kleisli structure $(\mathsf{S}, \epsilon, \vartheta)$ such that $\mathbb{X}$ is a Cartesian $k$-differential category where the projection maps $\pi_j$ are $\vartheta$-natural, $\mathsf{S}$ is a strong Cartesian $k$-differential functor, and for every object $A$, $\epsilon_A$ and $\vartheta_A$ are $\mathsf{D}$-linear. 
\end{definition}

We now explain why the subcategory of $\vartheta$-natural maps is a Cartesian $k$-differential category. 

\begin{lemma} Let $\mathbb{X}$ be a Cartesian $k$-differential abstract coKleisli category with abstract coKleisli structure $(\oc, \varphi, \vartheta)$. Then $\vartheta\text{-}\mathsf{nat}[\mathbb{X}]$ is a sub-Cartesian $k$-differential category of $\mathbb{X}$.
\end{lemma}
\begin{proof} It is immediate that the unique map to the terminal object $0: A \to \ast$ is always $\vartheta$-natural, and since the projections are $\vartheta$-natural, then the pairing of $\vartheta$-natural maps is again $\vartheta$-natural. Therefore, $\vartheta\text{-}\mathsf{nat}[\mathbb{X}]$ does indeed have finite products. Now using that $\mathsf{S}$ preserves $k$-linear structure and that since $\vartheta_A$ is $\mathsf{D}$-linear it is also $k$-linear, for any parallel $\vartheta$-natural maps $f$ and $g$ we compute: 
\begin{gather*}
    \mathsf{S}(r \cdot f + s\cdot g) \circ \vartheta_A = \left( r \cdot \mathsf{S}(f) + s \cdot \mathsf{S}(g) \right) \circ \vartheta_A = r \cdot \left(  \mathsf{S}(f) \circ \vartheta_A\right) +  s \cdot \left(  \mathsf{S}(g) \circ \vartheta_A\right) \\=  r \cdot \left(  \vartheta_B \circ f \right) +  s \cdot \left(  \vartheta_B \circ g \right) = \vartheta_A \circ (r \cdot f + s\cdot g)
\end{gather*}
 So $\vartheta$-natural maps are closed under $k$-linear structure, and we have that $\vartheta\text{-}\mathsf{nat}[\mathbb{X}]$ is a sub-Cartesian left $k$-linear category of $\mathbb{X}$. Lastly, we need to check that the derivative of a $\vartheta$-natural map is again $\vartheta$-natural map. To do so, first observe that by $\vartheta$-naturality of the projections, it follows that:
 \begin{align}
     \vartheta_{A_1} \times \hdots \times \vartheta_{A_j} = \omega_{A_1,\hdots, A_n} \circ \vartheta_{A_1 \times \hdots \times A_n}
 \end{align}
 Also recall that by \cite[Lem 2.2.2.(vi)]{blute2009cartesian}, if $g$ and $k$ are $\mathsf{D}$-linear then $\mathsf{D}[g \circ f] = g \circ \mathsf{D}[f]$ and $\mathsf{D}[f \circ k] = \mathsf{D}[f] \circ (k \times k)$. Thus since $\vartheta_A$ is $\mathsf{D}$-linear, we compute: 
 \begin{gather*}
    \vartheta_B \circ \mathsf{D}[f] = \mathsf{D}\left[ \vartheta_B \circ f  \right] = \mathsf{D}\left[  \mathsf{S}(f) \circ \vartheta_A \right] = \mathsf{D}\left[  \mathsf{S}(f)  \right]\circ (\vartheta_A \times \vartheta_A) \\
    = \mathsf{S}(\mathsf{D}[f]) \circ \omega^{-1}_{A,A} \circ (\vartheta_A \times \vartheta_A) =  \mathsf{S}(\mathsf{D}[f]) \circ \vartheta_{A \times A} 
 \end{gather*}
 So we conclude that $\vartheta\text{-}\mathsf{nat}[\mathbb{X}]$ is a Cartesian $k$-differential category.
\end{proof}

We can now prove that canonical monad on the subcategory of $\vartheta$-natural is a Cartesian differential monad, and in particular that the Kleisli category is isomorphic as a Cartesian differential category to the starting Cartesian differential abstract Kleisli category.  

\begin{proposition}\label{propab1} Let $\mathbb{X}$ be a Cartesian $k$-differential abstract Kleisli category with differential combinator $\mathsf{D}$ and abstract coKleisli structure $(\oc, \varphi, \vartheta)$. Then the monad $\mathbb{S}$, as defined in Lemma \ref{lem:ep-com}, is a Cartesian $k$-differential monad on $\vartheta\text{-}\mathsf{nat}[\mathbb{X}]$. Furthermore, the isomorphism $\mathsf{G}_\vartheta: \mathbb{X} \to \mathsf{KL}(\mathbb{S})$ is a strict Cartesian $k$-differential functor (and so is its inverse), thus $\mathbb{X} \cong \mathsf{KL}(\mathbb{S})$ as Cartesian $k$-differential categories. 
\end{proposition} 
\begin{proof} By assumption, $\mathsf{S}$ is a strong Cartesian $k$-differential functor and $\vartheta_A$ is $\mathsf{D}$-linear. So it remains to explain why $\delta_A$ is also $\mathsf{D}$-linear. However again by the assumption that $\epsilon_A$ is $\mathsf{D}$-linear and that, by Lem \ref{lem:dlinfunc}, $\mathsf{S}$ preserves $\mathsf{D}$-linear maps, it follows that $\delta_A = \mathsf{S}(\epsilon_A)$ is $\mathsf{D}$-linear. So $\mathbb{S}$ is indeed a Cartesian differential monad. So now we check that $\mathsf{G}_\vartheta$ is a strict Cartesian $k$-differential functor. Clearly, $\mathsf{G}_\vartheta$ is a strict Cartesian $k$-linear functor, so it remains to check that: 
\begin{gather*}
    \llbracket \mathsf{D}_{\mathbb{S}}\left[ \mathsf{G}_\mathsf{S}(f) \right] \rrbracket = \mathsf{D}\left[ \llbracket \mathsf{G}_\mathsf{S}(f)  \rrbracket \right] = \mathsf{D}\left[ \mathsf{S}(f) \circ \vartheta_A \right] =  \mathsf{D}\left[ \mathsf{S}(f)  \right] \circ (\vartheta_A \times \vartheta_A)\\
    = \mathsf{S}(\mathsf{D}[f]) \circ \omega^{-1}_{A,A} \circ (\vartheta_A \times \vartheta_A) = \mathsf{S}(\mathsf{D}[f]) \circ \vartheta_{A \times A} = \llbracket  \mathsf{G}_\mathsf{S}\left( \mathsf{D}\left[f \right] \right) \rrbracket 
\end{gather*}
So $\mathsf{G}_\vartheta$ is a strict Cartesian $k$-differential functor, and therefore so is $\mathsf{G}^{-1}_\vartheta$. Thus we conclude that $\mathbb{X} \cong \mathsf{KL}(\mathbb{S})$ are isomorphic as Cartesian $k$-differential categories. 
\end{proof}

We now wish to show the converse and explain how every Kleisli category of a Cartesian differential monad is a Cartesian differential abstract Kleisli category. First, let us recall how every Kleisli category for any monad is an abstract Kleisli category: 

\begin{lemma}\label{lem:kleisliAB} \cite[Sec 2]{fuhrmann1999direct} Let $\mathbb{S}$ be a monad on a category $\mathbb{X}$. Define the functor $\mathsf{S}_\mathbb{S}: \mathsf{KL}(\mathbb{S}) \to \mathsf{KL}(\mathbb{S})$ as the composite $\mathsf{S}_\mathbb{S} = \mathsf{L}_\mathbb{S} \circ \mathsf{R}_\mathbb{S}$, explicity, $\mathsf{S}_\mathbb{S}$ is defined on objects as $\mathsf{S}_\mathbb{S}(A) = \mathsf{S}(A)$ and on maps as:
\begin{align}
    \llbracket \mathsf{S}_\mathbb{S}(f) \rrbracket = \eta_{\mathsf{S}(A)} \circ \mu_A \circ \mathsf{S}\left( \llbracket f \rrbracket \right)
\end{align}
For every object $A$, define the maps $\epsilon_A \in \mathsf{KL}(\mathbb{S})(\mathsf{S}(A),A )$ and $\vartheta_A \in \mathsf{KL}(\mathbb{S})(A, \mathsf{S}(A) )$ as follows:
\begin{align}
    \llbracket \epsilon_A \rrbracket = 1_{\mathsf{S}(A)} && \llbracket \vartheta_A \rrbracket = \eta_{\mathsf{S}(A)} \circ \eta_A
\end{align}
Then $\mathsf{KL}(\mathbb{S})$ is an abstract Kleisli category with abstract Kleisli structure $(\mathsf{S}_\mathbb{S}, \epsilon, \vartheta )$.
\end{lemma}

For an arbitrary monad $\mathbb{S}$, $\vartheta\text{-}\mathsf{nat}[\mathsf{KL}(\mathbb{S})]$ is not necessarily isomorphic to $\mathbb{X}$. That said, $\mathbb{X}$ always embeds in $\vartheta\text{-}\mathsf{nat}[\mathsf{KL}(\mathbb{S})]$ since for any map $f$ in $\mathbb{X}$, $\mathsf{L}_\mathbb{S}(f)$ is $\vartheta$-natural in $\mathsf{KL}(\mathbb{S})$. If the monad $\mathbb{S}$ is exact, meaning that $\eta_A$ is an equalizer of $\eta_{\mathsf{S}(A)}$ and $\mathsf{S}(\eta_A)$, then we do have the isomorphism $\vartheta\text{-}\mathsf{nat}[\mathsf{KL}(\mathbb{S})] \cong \vartheta\text{-}\mathsf{nat}[\mathsf{KL}(\mathbb{S})]$ \cite[Sec 5]{fuhrmann1999direct}. For an abstract Kleisli category, the canonical comonad on its subcategory of $\vartheta$-natural maps is exact. 

\begin{proposition}\label{propab2} Let $\mathbb{S}$ be a Cartesian $k$-differential monad on a Cartesian $k$-differential category $\mathbb{X}$. Then $\mathsf{KL}(\mathbb{S})$ is a Cartesian $k$-differential abstract Kleisli category with abstract Kleisli structure $(\mathsf{S}_\mathbb{S}, \epsilon, \vartheta )$ defined as in Lemma \ref{lem:kleisliAB}. 
\end{proposition}
\begin{proof} As explained above, since the projections are of the form $\llbracket \pi_j \rrbracket = \llbracket \mathsf{L}_\mathbb{S}(\pi_j) \rrbracket$, they are $\vartheta$-natural. Furthermore, since $\mathsf{S}_\mathbb{S} = \mathsf{L}_\mathbb{S} \circ \mathsf{R}_\mathbb{S}$ is the composite of strong Cartesian $k$-differential functors, it is itself a strong Cartesian $k$-differential functor. Lastly, since $\llbracket \epsilon_A \rrbracket$ and $\llbracket \vartheta_A \rrbracket$ are both $\mathsf{D}$-linear in $\mathbb{X}$, since identity maps are $\mathsf{D}$-linear and the composition of $\mathsf{D}$-linear maps is again $\mathsf{D}$-linear \cite[Lem 2.2.2.(iii)]{blute2009cartesian}, by Thm \ref{thm:CDMKleisli}, $\epsilon_A$ and $\vartheta_A$ are $\mathsf{D}_\mathbb{S}$-linear in $\mathsf{KL}(\mathbb{S})$. So we conclude that $\mathsf{KL}(\mathbb{S})$ is a Cartesian $k$-differential abstract Kleisli category, as desired.  
\end{proof}

Together, Prop \ref{propab1} \& \ref{propab2} give us a precise characterization of which Cartesian differential categories are Kleisli categories of Cartesian differential monads. So we conclude with the following: 

\begin{corollary} Let $\mathbb{Y}$ be a Cartesian $k$-differential category. Then there exists a Cartesian $k$-differential monad $\mathbb{S}$ on a Cartesian $k$-differential category $\mathbb{X}$ such that $\mathbb{Y} \cong \mathsf{KL}(\mathbb{S})$ are isomorphic as Cartesian $k$-differential categories if and only if $\mathbb{Y}$ is a Cartesian $k$-differential abstract Kleisli category. 
\end{corollary}

\section{Kleisli Differential Combinators}

In theory, it is possible to start with a Cartesian $k$-linear monad on a Cartesian left $k$-linear category, one possibly without a differential combinator, and ask when the Kleisli category has a differential combinator. However, we will argue that in practice, this is not necessarily more efficient. Nevertheless, for completeness, let us introduce the notion of a Kleisli differential combinator, which is a direct description of a differential combinator in the Kleisli category of a Cartesian $k$-linear monad. 

\begin{definition}\label{def:KDC} Let $\mathbb{S}$ be a Cartesian $k$-linear monad on a Cartesian left $k$-linear category $\mathbb{X}$. A \textbf{Kleisli differential combinator} on $\mathbb{S}$ is a family of functions $\mathsf{B}: \mathbb{X}(A,B) \to \mathbb{X}(A \times A,\mathsf{S}(B))$, where for a map $f: A \to B$, $\mathsf{B}[f]: A \times A \to \mathsf{S}(B)$ is called the \textbf{Kleisli derivative} of $f$, and such that the following axioms hold:  
\begin{description}
    \item[{\bf [KD.$\mathsf{S}$]}]  $\mathsf{B}[\mathsf{S}(f)] = \mathsf{S}(\mathsf{B}[f]) \circ \omega^{-1}_{A,A}$;  
 \item[{\bf [KD.$\mu$]}] $\mathsf{B}[\mu_A] = \eta_{\mathsf{S}(A)} \circ \mu_A \circ \pi_1$;
\item[{\bf [KD.$\eta$]}] $\mathsf{B}[\eta_A] = \eta_{\mathsf{S}(A)} \circ \eta_A \circ \pi_1$; 
\end{description}
\begin{enumerate}[{\bf [KD.1]}]
\item \label{KDCax1} $\mathsf{B}[r\cdot f+ s\cdot g] = r\cdot \mathsf{B}[f] + s\cdot \mathsf{B}[g]$;
\item \label{KDCax2} $\mathsf{B}[f] \circ \langle \pi_1, r\cdot \pi_2 + s \cdot \pi_3 \rangle = r\cdot \mathsf{B}[f] \circ \langle \pi_1, \pi_2 \rangle + s \cdot \mathsf{B}[f] \circ \langle \pi_1, \pi_3 \rangle$;  
\item \label{KDCax3} $\mathsf{B}[1_A]=\eta_A \circ \pi_1$ and $\mathsf{B}[\pi_j] = \eta_{A_{n+j}} \circ \pi_{n+j}$;
\item \label{KDCax4} $\mathsf{B}[\langle f_1, \hdots, f_n \rangle] =  \omega^{-1}_{B_1,\hdots,B_n} \circ \langle \mathsf{B}[f_1] , \hdots, \mathsf{B}[f_n] \rangle$; 
\item \label{KDCax5} $\mathsf{B}[g \circ f] = \mu_C \circ \mathsf{S}\left( \mathsf{B}[g] \right)   \circ  \omega^{-1}_{B,B}  \circ \langle \eta_B \circ f \circ \pi_0, \mathsf{B}[f] \rangle$; 
\item \label{KDCax6} $\mathsf{B}\!\left[\mathsf{B}[f] \right] \circ \left \langle \pi_1, 0, 0, \pi_2 \right \rangle = \eta_{\mathsf{S}(B)} \circ \mathsf{B}[f] \circ \langle x, y \rangle $ 
\item \label{KDCax7} $\mathsf{B}\!\left[\mathsf{B}[f] \right] \circ \left \langle \pi_1, \pi_2, \pi_3, \pi_4 \right \rangle  = \mathsf{B}\!\left[\mathsf{B}[f] \right] \circ \left \langle \pi_1, \pi_3, \pi_2, \pi_4 \right \rangle$  
\end{enumerate}
\end{definition}

The axioms {\bf [KD.1]} to {\bf [KD.7]} of a Kleisli differential combinator are just what the axioms of a differential combinator in the Kleisli category would be but expressed in the base category. As such, a Kleisli differential combinator does indeed give us a differential combinator: 

\begin{proposition} Let $\mathbb{X}$ be a Cartesian left $k$-linear category and let $\mathbb{S}$ be a Cartesian $k$-linear monad on $\mathbb{X}$ that is equipped with a Kleisli differential combinator $\mathsf{B}$. Then the Kleisli category $\mathsf{KL}(\mathbb{S})$ is a Cartesian $k$-differential category where the Cartesian left $k$-linear structure is defined as in Lemma \ref{lem:KleisliCLAC} and the differential combinator $\mathsf{D}_\mathsf{B}$ is defined on a Kleisli map $\llbracket f \rrbracket: A \to \mathsf{S}(B)$ as follows:
\begin{align}
    \llbracket \mathsf{D}_\mathsf{B}[f] \rrbracket := \mu_B \circ \mathsf{B}\left[ \llbracket f \rrbracket \right]
\end{align}
Furthermore, $\mathsf{KL}(\mathbb{S})$ is a Cartesian $k$-differential abstract Kleisli category with abstract Kleisli structure $(\mathsf{S}_\mathbb{S}, \epsilon, \vartheta )$ defined as in Lemma \ref{lem:kleisliAB}. 
\end{proposition}
\begin{proof} The proof that $\mathsf{D}_\mathsf{B}$ is indeed a differential combinator is essentially by brute force calculations, and not necessarily more enlightening for the story of this paper. As such, we will leave it as an exercise for the motivated reader. Using {\bf [KD.$\mathsf{S}$]}, {\bf [KD.$\mu$]}, and {\bf [KD.$\eta$]} it follows that $\mathsf{KL}(\mathbb{S})$ is indeed a Cartesian $k$-differential abstract Kleisli category as well.  
\end{proof}

Every Cartesian differential monad has a canonical Kleisli differential combinator given by simply composing derivatives by the unit: 

\begin{proposition} Let $\mathbb{S}$ be a Cartesian $k$-differential monad on a Cartesian $k$-differential category $\mathbb{X}$. Then $\mathbb{S}$ has a Kleisli differential combinator $\mathsf{B}_\mathbb{S}$ defined on a map $f: A \to B$ as:
\begin{align}
    \mathsf{B}_\mathsf{D}[f] := \eta_B \circ \mathsf{D}\left[ f \right] 
\end{align}
Furthermore, the differential combinators $\mathsf{D}_\mathbb{S}$ and $\mathsf{D}_{\mathsf{B}_\mathbb{S}}$ on the Kleisli category are equal:
\begin{align}
   \llbracket \mathsf{D}_{\mathsf{B}_\mathbb{S}}[f] \rrbracket=  \llbracket \mathsf{D}_\mathbb{S}[f] \rrbracket  
\end{align}
\end{proposition}
\begin{proof} We leave checking the Kleisli differential combinator axioms as an exercise for the motivated reader, as they are straightforward to check. That the differential combinators are the same follows from the monad axiom: 
\[ \llbracket \mathsf{D}_{\mathsf{B}_\mathbb{S}}[f] \rrbracket = \mu_B \circ \mathsf{B}_\mathsf{D}\left[ \llbracket f \rrbracket \right] = \mu_B \circ \eta_B \circ \mathsf{D}\left[ \llbracket f \rrbracket \right] = \mathsf{D}\left[ \llbracket f \rrbracket \right]  = \llbracket \mathsf{D}_\mathbb{S}[f] \rrbracket   \]
So $\mathsf{D}_{\mathsf{B}_\mathbb{S}}$ is precisely $\mathsf{D}_\mathbb{S}$ as desired. 
\end{proof}

While Kleisli differential combinators allow one to start from a slightly more general setting, we will now explain why they are not necessarily more advantageous than Cartesian differential monads. Indeed, using results from the previous section we can explain why every Kleisli differential combinator induces a Cartesian differential monad whose Kleisli categories are isomorphic. 

So let $\mathbb{S}$ be a Cartesian $k$-linear monad on $\mathbb{X}$ that is equipped with a Kleisli differential combinator. Then by Prop \ref{propab1} \& \ref{propab2}, we have a Cartesian differential monad $\mathbb{S}_\mathbb{S}$ on $\vartheta\text{-}\mathsf{nat}[\mathsf{KL}(\mathbb{S})]$ such that $\mathsf{KL}(\mathbb{S}) \cong \mathsf{KL}(\mathbb{S}_\mathbb{S})$ are isomorphic as Cartesian $k$-differential categories. Thus we could have simply started from a Cartesian differential monad (albeit on a possibly more complicated category). Furthermore, if $\mathbb{S}$ is exact, then $\vartheta\text{-}\mathsf{nat}[\mathsf{KL}(\mathbb{S})] \cong \mathbb{X}$, which implies that $\mathbb{X}$ was actually a Cartesian $k$-differential category to start with and the Kleisli differential combinator is just given by post-composing the differential combinator by the unit. So in conclusion we have that: 

\begin{lemma} Let $\mathbb{X}$ be a Cartesian left $k$-linear category and let $\mathbb{S}$ be a Cartesian $k$-linear monad on $\mathbb{X}$ that is equipped with a Kleisli differential combinator $\mathsf{B}$. Then there exists a Cartesian $k$-differential monad $\mathbb{S}^\prime$ on a Cartesian $k$-differential category $\mathbb{X}^\prime$ such that the Kleisli categories $\mathsf{KL}(\mathbb{S}) \cong \mathsf{KL}(\mathbb{S}^\prime)$ are isomorphic as Cartesian $k$-differential categories. 
\end{lemma}

For the sake of completeness, here is a trivial example which shows that it is possible to have a Kleisli differential combinator which is not directly induced from a Cartesian differential monad: 

\begin{example} Trivially, for any Cartesian left $k$-linear category, the constant monad has a Kleisli differential combinator which sends any map to zero. The Kleisli category in this case is again simply isomorphic to the terminal category.  
\end{example}

It would be interesting to find more examples of Kleisli differential combinators on Cartesian left $k$-linear categories which are not Cartesian $k$-differential categories. Furthermore, it seems likely that Kleisli differential combinators play a role in the story of \cite{bauer2018directional} and thus could help provide a deeper understanding of what makes this model work. In turn, Kleisli differential combinators could potentially be used to help provide a general construction of differential combinators for homotopy categories of Kleisli categories. 

\section{Eilenberg-Moore Categories of Cartesian Differential Monads}\label{sec:tancats}

Whenever one has a monad, an important question to ask is what the algebras of the monad and what is its Eilenberg-Moore category. In this section, we will explain why for a Cartesian differential monad, its Eilenberg-Moore category is, in fact, a tangent category \cite{cockett2014differential}. We will only provide a brief overview of tangent categories, so we invite the reader to see \cite{cockett2014differential} for a more detailed introduction. 

Recall that for a monad $\mathbb{S}$ on a category $\mathbb{X}$, an $\mathbb{S}$-algebra is a pair $(A, \alpha)$ consisting of an object $A$ and a map $\alpha: \mathsf{S}(A) \to A$ such that:
\begin{align}
    \alpha \circ \mu_A = \alpha \circ \mathsf{S}(\alpha) && \alpha \circ \eta_A = 1_A
\end{align}
and that an $\mathbb{S}$-algebra morphism $f: (A, \alpha) \to (B,\beta)$ is a map $f: A \to B$ such that:
\begin{align}
  \beta \circ \mathsf{S}(f) = f \circ \alpha  
\end{align}
The category of $\mathbb{S}$-algebras and $\mathbb{S}$-algebra morphisms between them is called the Eilenberg-Moore category of $\mathbb{S}$, and is denoted as $\mathsf{EM}(\mathbb{S})$. There is a canonical forgetful functor $\mathsf{U}_\mathbb{S}: \mathsf{EM}(\mathbb{S}) \to \mathbb{X}$ which is defined on objects as $\mathsf{U}_\mathbb{S}(A, \alpha) = A$ and on maps as $\mathsf{U}_\mathbb{S}(f) = f$. 

It is well-known that if $\mathbb{X}$ has finite products, then $\mathsf{EM}(\mathbb{S})$ also has finite products where in particular:
\begin{align}
   (A,\alpha) \times (B, \beta) = (A \times B, (\alpha \times \beta) \circ \omega_{A,B}) 
\end{align}
However if $\mathbb{X}$ is a Cartesian left $k$-linear category, and even if $\mathbb{S}$ is a Cartesian $k$-linear monad, in general, $\mathsf{EM}(\mathbb{S})$ may not be a Cartesian left $k$-linear category. This is because $\mathbb{S}$-algebra morphisms are not necessarily closed under $k$-linear structure. Indeed, $\beta \circ \mathsf{S}(r\cdot f + s \cdot g)$ may not equal $(r\cdot f + s \cdot g) \circ \alpha$. Similarly, if $\mathbb{X}$ is a Cartesian $k$-differential category, even if $\mathbb{S}$ is a Cartesian $k$-differential monad, then the derivative of an $\mathbb{S}$-algebra morphism may not be an $\mathbb{S}$-algebra morphism. So while $\mathsf{EM}(\mathbb{S})$ is not necessarily a Cartesian $k$-differential category, it turns out that it is still a tangent category. So let us briefly recall the necessary structure: 

\begin{definition} \cite[Def 2.8]{cockett2014differential} A \textbf{Cartesian tangent category} is a category $\mathbb{X}$ with finite products, equipped with:
\begin{enumerate}[{\em (i)}]
\item A family of functors $\mathsf{T}_n: \mathbb{X} \to \mathbb{X}$ (for each $n \in \mathbb{N}$), where by convention $\mathsf{T}_0 = 1_\mathbb{X}$ and $\mathsf{T}_1 = \mathsf{T}$ is called the \textbf{tangent bundle functor};
\item A natural transformation $\mathsf{p}_A: \mathsf{T}(A) \to A$ called the projection; 
\item A natural transformation $\mathsf{s}_A: \mathsf{T}_2(A) \to \mathsf{T}(A)$ called the sum;
\item A natural transformation $\mathsf{z}_A: A \to \mathsf{T}(A)$ called the zero;
\item A natural transformation $\ell_A: \mathsf{T}(A) \to \mathsf{T}^2(A)$ called the vertical lift; 
\item A natural transformation $\mathsf{c}_A: \mathsf{T}^2(A) \to \mathsf{T}(A)$ called the canonical flip; 
\end{enumerate}
 such that all the necessary axioms in \cite[Def 2.3 \& 2.8]{cockett2014differential} hold. 
\end{definition}

Tangent categories formalize the properties of the tangent bundle on smooth manifolds from classical differential geometry. As such, an object $A$ in a tangent category can be interpreted as a base space, and $\mathsf{T}(A)$ as its abstract tangent bundle, where $\mathsf{T}_n(A)$ is the $n$-fold fibres of the tangent bundle. The natural transformations all encode essential properties of the tangent bundle such as the natural projection, the fact that fibre is a vector bundle, etc. 

Every Cartesian differential category $\mathbb{X}$ is a Cartesian tangent category \cite[Prop 4.7]{cockett2014differential}, where in particular, the tangent bundle functor is defined as in Ex \ref{ex:tanbun}, $\mathsf{T}_n(A)$ is the product of $n+1$ copies of $A$, and the rest of the natural transformations are defined as follows: 
 \begin{equation}\begin{gathered}   \mathsf{p}_A := \pi_1 \quad \quad \quad \mathsf{s}_A := \langle \pi_1, \pi_2 + \pi_3 \rangle \quad \quad \quad \mathsf{z}_A := \langle 1_A, 0 \rangle \\
    \ell_A := \langle \pi_1, 0, 0, \pi_2 \rangle \quad \quad \quad \mathsf{c}_A := \langle \pi_1, \pi_3, \pi_2, \pi_4 \rangle \end{gathered}\end{equation}
So how can we explain why the Eilenberg-Moore category of a Cartesian differential monad is a tangent category? The answer is again by lifting, where this time we wish to lift the tangent structure of the base category to the Eilenberg-Moore category. 

A \textbf{tangent monad} \cite[Def 19]{cockett_et_al:LIPIcs:2020:11660} on a tangent category is precisely the kind of monad which lifts the tangent structure to its Eilenberg-Moore category. A bit more explicitly, a tangent monad on a tangent category $\mathbb{X}$ is a monad $\mathbb{S}$ on $\mathbb{X}$ which also comes equipped with a natural transformation ${\lambda_A: \mathsf{S}\mathsf{T}(A) \to \mathsf{T}\mathsf{S}(A)}$ which is compatible with $\mu$ and $\eta$, and such that $\lambda$ is also compatible with the tangent structure in the sense of \cite[Def 2.7]{cockett2014differential}. By \cite[Prop 20]{cockett_et_al:LIPIcs:2020:11660}, the Eilenberg-Moore category of a tangent monad on a Cartesian tangent category is also a Cartesian tangent category where the tangent bundle on an $\mathbb{S}$-algebra is $\mathsf{T}(A, \alpha) = (\mathsf{T}(A), \mathsf{T}(\alpha) \circ \lambda_A)$. Furthermore, the forgetful functor $\mathsf{U}_\mathbb{S}: \mathsf{EM}(\mathbb{S}) \to \mathbb{X}$ preserves the Cartesian tangent structure strictly. It turns out that Cartesian differential monads are always tangent monads: 

\begin{proposition}  Let $\mathbb{S}$ be a Cartesian $k$-differential monad on a Cartesian $k$-differential category $\mathbb{X}$. Then $\mathbb{S}$ is a tangent monad where $\lambda_A := \omega_{A,A}: \mathsf{S}\mathsf{T}(A) \to \mathsf{T}\mathsf{S}(A)$.
\end{proposition}
\begin{proof} We need to explain why $\lambda$ is indeed a natural transformation. So we compute: 
\begin{gather*}
    \lambda_B \circ \mathsf{S}\mathsf{T}(f) = \omega_{B,B} \circ \mathsf{S}\left( \left \langle f \circ \pi_1, \mathsf{D}[f] \right \rangle \right) = \omega_{B,B} \circ \omega^{-1}_{B,B} \circ \left \langle \mathsf{S}(f \circ \pi_1), \mathsf{S}( \mathsf{D}[f] ) \right \rangle = \left \langle \mathsf{S}(f) \circ \mathsf{S}(\pi_1), \mathsf{S}( \mathsf{D}[f] ) \right \rangle \\
    = \left \langle \mathsf{S}(f) \circ \pi_1 \circ \omega_{A,A}, \mathsf{D}[\mathsf{S}( f )] \circ \omega_{A,A} \right \rangle = \left \langle \mathsf{S}(f) \circ \pi_1 , \mathsf{D}[\mathsf{S}( f )]  \right \rangle \circ \omega_{A,A} = \mathsf{T}\mathsf{S}(f) \circ \lambda_A
\end{gather*}
It is also straightforward to check that the necessary identities in \cite[Def 19]{cockett_et_al:LIPIcs:2020:11660} hold. So we conclude that $\mathbb{S}$ is a tangent monad. 
\end{proof}

Therefore by \cite[Prop 20]{cockett_et_al:LIPIcs:2020:11660}, we have the main result of this section: 

\begin{theorem}\label{thm:EMCDM} Let $\mathbb{S}$ be a Cartesian $k$-differential monad on a Cartesian $k$-differential category $\mathbb{X}$. Then $\mathsf{EM}(\mathbb{S})$ is a Cartesian tangent category such that the $\mathsf{U}_\mathbb{S}: \mathsf{EM}(\mathbb{S}) \to \mathbb{X}$ preserves the Cartesian tangent structure strictly. In particular, the tangent bundle of an $\mathsf{S}$-algebra is given by:
\begin{align}
    \mathsf{T}(A, \alpha) = (A \times A, \left \langle \alpha \circ \pi_1, \mathsf{D}[\alpha] \circ \omega_{A,A} \right \rangle
\end{align}
\end{theorem}

Thus algebras of Cartesian differential monads can be interpreted as abstract smooth manifolds. 

\begin{example} Trivially, for the identity monad, its Eilenberg-Moore category is just the start Cartesian differential category, while for the constant monad, its Eilenberg-Moore category is again equivalent to the terminal category. 
\end{example}

\begin{example} As explained in \cite{alvarez2020cartesian,cockett2014differential}, for both Cartesian differential categories and tangent categories, algebras of the tangent bundle monad are not necessarily well studied and somewhat mysterious. Even in specific examples from differential calculus and differential geometry, there does not in general seem to be a precise characterization beyond the categorical definition. Yet, the above theorem tells us that the Eilenberg-Moore category of the tangent bundle monad is always a Cartesian tangent category -- which may have interesting consequences in future work. 
\end{example}

\begin{example} In general, it is understood that algebras of reader monads are not easily characterized. However, we now know that for Cartesian closed differential categories, and thus for any model of the differential $\lambda$-calculus, the algebras of the reader monads form a tangent category. 
\end{example}

In a Cartesian tangent category, there is an important kind of object called a \textbf{differential object} \cite[Def 4.8]{cockett2014differential}, which are analogues of Euclidean spaces. They are important since the subcategory of differential objects is a Cartesian differential category \cite[Thm 4.11]{cockett2014differential}. Briefly, a differential object is a commutative monoid $A$, with binary operation $\sigma: A \times A \to A$ and unit $\zeta: \ast \to A$, which also comes equipped a map $\hat{\mathsf{p}}: \mathsf{T}(A) \to A$ such that the necessary axioms in \cite[Def 4.8]{cockett2014differential} hold. In particular, the key feature of a differential object is that $\mathsf{T}(A) \cong A \times A$. If $\mathbb{X}$ is a Cartesian tangent category, let $\mathsf{DIFF}[\mathbb{X}]$ be the subcategory of differential objects and all maps between them. Then, as mentioned, $\mathsf{DIFF}[\mathbb{X}]$ is a Cartesian differential category where the derivative of a map $f: A \to B$ is defined by the composite of the isomorphism $A \times A \cong \mathsf{T}(A)$ followed by $\mathsf{T}(f)$ followed by $\hat{\mathsf{p}}$. 

In a Cartesian $k$-differential category $\mathbb{X}$, every object has a canonical and unique differential object structure \cite[Prop 4.7]{cockett2014differential} since in particular $\mathsf{T}(A) = A \times A$. Concretely, 
\begin{align}
    \hat{\mathsf{p}} = \pi_2 && \sigma = \pi_1 + \pi_2 && \zeta = 0
\end{align}
So in this case, $\mathsf{DIFF}[\mathbb{X}] = \mathbb{X}$, and the resulting differential combinator is the same as the starting one. 

For a Cartesian differential monad, we will now explain how its Kleisli category embeds as a Cartesian $k$-differential category into the subcategory of differential objects of its Eilenberg-Moore category. For starters, let us describe the differential objects. So let $\mathbb{S}$ be a Cartesian $k$-differential monad on a Cartesian $k$-differential category $\mathbb{X}$. Since the $\mathsf{U}_\mathbb{S}: \mathsf{EM}(\mathbb{S}) \to \mathbb{X}$ preserves the Cartesian tangent structure strictly, it also preserves differential objects. Thus a differential object in $\mathsf{EM}(\mathbb{S})$ must also be one in $\mathbb{X}$. However, this implies that an $\mathbb{S}$-algebra $(A, \alpha)$ can have at most one differential object structure given precisely as above. In other words, if: 
\begin{align*}
    \pi_2: \mathsf{T}(A, \alpha) \to (A, \alpha) && \pi_1 + \pi_2: (A, \alpha) \times (A, \alpha) \to (A, \alpha) && 0: (\ast, 0) \to (A, \alpha)
\end{align*}
are $\mathbb{S}$-algebra morphisms, then $(A, \alpha)$ is a differential object. Furthermore, the induced isomorphism $\mathsf{T}(A, \alpha) \cong (A, \alpha) \times (A, \alpha)$ must be an equality on the nose, i.e. $\mathsf{T}(A, \alpha)= (A, \alpha) \times (A, \alpha)$. It easy to check that $\mathsf{T}(A, \alpha)= (A, \alpha) \times (A, \alpha)$ will actually imply that $\pi_2$, $\pi_1 + \pi_2$, and $0$ are all $\mathbb{S}$-algebra morphisms as desired. So we have that an $\mathbb{S}$-algebra is a differential object if and only if $\mathsf{T}(A, \alpha)= (A, \alpha) \times (A, \alpha)$. It turns out that this is precisely the case when the $\mathbb{S}$-algebra is differential linear. 

\begin{lemma}\label{lem:diffobjCDM} Let $\mathbb{S}$ be a Cartesian $k$-differential monad on a Cartesian $k$-differential category $\mathbb{X}$. Then an $\mathbb{S}$-algebra $(A, \alpha)$ is a differential object in the Eilenberg-Moore category $\mathsf{EM}(\mathbb{S})$ if and only if $\alpha$ is $\mathsf{D}$-linear. 
\end{lemma}
\begin{proof} First observe that $\mathsf{T}(A, \alpha)= (A, \alpha) \times (A, \alpha)$ if and only if $\mathsf{T}(\alpha) \circ \omega_{A,A} = (\alpha \times \alpha) \circ \omega_{A,A}$. Since $\omega$ is a natural isomorphism, we have that: $\mathsf{T}(A, \alpha)= (A, \alpha) \times (A, \alpha)$ if and only if $\mathsf{T}(\alpha) = \alpha \times \alpha$. However, a map $f$ is $\mathsf{D}$-linear if and only if $\mathsf{T}(f) = f \times f$. Thus $\mathsf{T}(A, \alpha)= (A, \alpha) \times (A, \alpha)$ if and only if $\alpha$ is $\mathsf{D}$-linear. Thus, per the above discussion, $(A, \alpha)$ is a differential object if and only if $\alpha$ is $\mathsf{D}$-linear. 
\end{proof}

As a result, we get that the Cartesian $k$-differential structure for the subcategory of differential objects is the same as the base Cartesian $k$-differential category. In particular, this means that the differential combinator lifts to the subcategory of differential objects.  

\begin{corollary} Let $\mathbb{S}$ be a Cartesian $k$-differential monad on a Cartesian $k$-differential category $\mathbb{X}$. Then $\mathsf{DIFF}[\mathsf{EM}(\mathbb{S})]$ is a Cartesian $k$-differential category such that the obvious forgetful functor $\mathsf{DIFF}[\mathsf{EM}(\mathbb{S})] \to \mathbb{X}$ is a strict Cartesian $k$-differential functor. So in particular, if $(A,\alpha)$ and $(B, \beta)$ are differential object, then for every $\mathbb{S}$-algebra morphism $f: (A, \alpha) \to (B, \beta)$, its derivative is also an $\mathbb{S}$-algebra morphism $\mathsf{D}[f]: (A, \alpha) \times (A, \alpha) \to (B, \beta)$. 
\end{corollary}
\begin{proof} It may be useful to directly workout why if $(A,\alpha)$ and $(B, \beta)$ are differential object, then for every $\mathbb{S}$-algebra morphism $f: (A, \alpha) \to (B, \beta)$, $\mathsf{D}[f]: (A, \alpha) \times (A, \alpha) \to (B, \beta)$ is also an $\mathbb{S}$-algebra morphism. Since $\alpha$ and $\beta$ are $\mathsf{D}$-linear, we may use \cite[Lem 2.2.2.(vi)]{blute2009cartesian} and also (\ref{SD}) to compute: 
 \begin{gather*}
\mathsf{D}[f] \circ (\alpha \times \alpha) \circ \omega_{A,A} = \mathsf{D}[f \circ \alpha] \circ \omega_{A,A} = \mathsf{D}[\beta \circ \mathsf{S}(f)] \circ \omega_{A,A} = \beta \circ \mathsf{D}[\mathsf{S}(f)] \circ \omega_{A,A} = \beta \circ \mathsf{S}\left(\mathsf{D}[f]  \right) 
 \end{gather*}
  So $\mathsf{D}[f]: (A, \alpha) \times (A, \alpha) \to (B, \beta)$ is indeed an $\mathbb{S}$-algebra morphism.
\end{proof}

We conclude with the statement that for a Cartesian differential monad, the canonical embedding of the Kleisli category into the Eilenberg-Moore category factors through the subcategory of differential objects. 

\begin{proposition} Let $\mathbb{S}$ be a Cartesian $k$-differential monad on a Cartesian $k$-differential category $\mathbb{X}$. Then for every object $A$, $(\mathsf{S}(A), \mu_A)$ is a differential object in the Eilenberg-Moore category $\mathsf{EM}(\mathbb{S})$. Furthermore, the functor $\mathsf{E}_\mathbb{S}: \mathsf{KL}(\mathbb{S}) \to \mathsf{DIFF}[\mathsf{EM}(\mathbb{S})]$, defined on objects as $\mathsf{E}_\mathbb{S}(A) = (\mathsf{S}(A), \mu_A)$ and on maps as ${\mathsf{E}_\mathbb{S}(f) = \mathsf{S}\left( \llbracket f \rrbracket \right)}$ is a strong Cartesian $k$-differential functor that is full and faithful. 
\end{proposition}
\begin{proof} For any object $A$, $(\mathsf{S}(A), \mu_A)$ is always a $\mathbb{S}$-algebra. Since $\mu_A$ is also assumed to be $\mathsf{D}$-linear, it follows from Lem \ref{lem:diffobjCDM} that $(\mathsf{S}(A), \mu_A)$ is a differential object. So $\mathsf{E}_\mathbb{S}$ is well defined. Since $\mathsf{S}$ is a strong Cartesian $k$-differential functor, it follows that $\mathsf{E}_\mathbb{S}$ is as well, and it is clear that it is full and faithful. 
\end{proof}

\section{Conclusion}

In this paper, we introduced Cartesian differential monads and showed that their Kleisli categories are Cartesian differential categories such that the differential combinator is lifted from the base category to the Kleisli category. However, the requirement that the monad's functor preserves the Cartesian $k$-linear structure is somewhat of a strong requirement. While some interesting monads do satisfy this requirement, such as tangent bundle monads and reader monads, there are many important monads which do not, such as state monads, continuation monads, and writer monads. Instead, we conjecture that while the Kleisli categories of these monads may not be Cartesian differential categories, their Kleisli categories will still be tangent categories. Indeed, for some of these monads, there is possibly a distributive law between it and the tangent bundle monad, which would lift the tangent structure of the base Cartesian differential category to the Kleisli category. Thus effectful programs may still form a tangent category, further motivating the need for a term calculus for tangent categories and developing ``tangent programming languages''. 

\section*{Acknowledgement}
The author would like to thank Jonathan Gallagher, Kristine Bauer, Robin Cockett, Geoff Cruttwell, Neil Ghani, and Masahito Hasegawa for useful discussions. For this research, the author was financially supported by an NSERC PDF (456414649), a JSPS PDF (P21746), and an ARC DECRA (DE230100303).

\bibliographystyle{./entics}     % mathematics and physical sciences
\bibliography{kleislibib}   % name your BibTeX data base

\begin{thebibliography}{10}
\providecommand{\url}[1]{\texttt{#1}}
\providecommand{\urlprefix}{ }
\providecommand{\eprint}[2][]{\url{#2}}

\bibitem{abadi2019simple}
Abadi, M. and G.~D. Plotkin, \emph{A simple differentiable programming
  language}, Proceedings of the ACM on Programming Languages \textbf{4}, pages
  1--28 (2019).
\newline\urlprefix\url{https://doi.org/10.1145/3190508.3190551}

\bibitem{alvarez2020cartesian}
Alvarez-Picallo, M. and J.-S.~P. Lemay, \emph{Cartesian difference categories},
  in: \emph{Foundations of Software Science and Computation Structures: 23rd
  International Conference, FOSSACS 2020, Held as Part of the European Joint
  Conferences on Theory and Practice of Software, ETAPS 2020, Dublin, Ireland,
  April 25--30, 2020, Proceedings 23}, pages 57--76, Springer (2020).
  \newline\urlprefix\url{https://doi.org/10.1007/978-3-030-45231-5_4}

\bibitem{lemay2021cartesian}
Alvarez-Picallo, M. and J.-S.~P. Lemay, \emph{Cartesian difference categories
  (extended version)}, Logical Methods in Computer Science \textbf{17} (2021).
    \newline\urlprefix\url{https://doi.org/10.46298/lmcs-17(3:23)2021}

\bibitem{alvarez2019change}
Alvarez-Picallo, M. and C.-H.~L. Ong, \emph{Change actions: models of
  generalised differentiation}, in: \emph{International Conference on
  Foundations of Software Science and Computation Structures}, pages 45--61,
  Springer (2019).
\newline\urlprefix\url{https://doi.org/10.1007/3-540-10286-8}

\bibitem{bauer2018directional}
Bauer, K., B.~Johnson, C.~Osborne, E.~Riehl and A.~Tebbe, \emph{Directional
  derivatives and higher order chain rules for abelian functor calculus},
  Topology and its Applications \textbf{235}, pages 375--427 (2018).
  \newline\urlprefix\url{https://doi.org/10.1016/j.topol.2017.12.010}

\bibitem{blute2009cartesian}
Blute, R.~F., J.~R.~B. Cockett and R.~A.~G. Seely, \emph{Cartesian differential
  categories}, Theory and Applications of Categories \textbf{22}, pages
  622--672 (2009). Available online at \url{http://www.tac.mta.ca/tac/volumes/22/23/22-23abs.html}

\bibitem{blute2015cartesian}
Blute, R.~F., J.~R.~B. Cockett and R.~A.~G. Seely, \emph{Cartesian differential
  storage categories}, Theory and Applications of Categories \textbf{30}, pages
  620--686 (2015).
  \newline\urlprefix\url{https://doi.org/10.1016/S0304-3975(01)00243-2}

\bibitem{bucalo2003equational}
Bucalo, A., C.~F{\"u}hrmann and A.~Simpson, \emph{An equational notion of
  lifting monad}, Theoretical Computer Science \textbf{294}, pages 31--60
  (2003).

\bibitem{cockettetal:LIPIcs:2020:11661}
Cockett, J. R.~B., G.~S.~H. Cruttwell, J.~D. Gallagher, J.-S.~P. Lemay,
  B.~MacAdam, G.~Plotkin and D.~Pronk, \emph{Reverse derivative categories},
  LIPIcs \textbf{152}, pages 18:1--18:16 (2020).
\newline\urlprefix\url{https://doi.org/10.4230/LIPIcs.CSL.2020.18}

\bibitem{Cockett-2019}
Cockett, J. R.~B. and J.~D. Gallagher, \emph{Categorical models of the
  differential $\lambda$-calculus}, Mathematical Structures in Computer Science
   (2019).
\newline\urlprefix\url{https://doi.org/10.1017/S0960129519000070}

\bibitem{cockett2014differential}
Cockett, R. and G.~Cruttwell, \emph{Differential structure, tangent structure,
  and {SDG}}, Applied Categorical Structures \textbf{22}, pages 331--417
  (2014).
  \newline\urlprefix\url{https://doi.org/10.1007/s10485-013-9312-0}

\bibitem{cockett_et_al:LIPIcs:2020:11660}
Cockett, R., J.-S.~P. Lemay and R.~Lucyshyn-Wright, \emph{Tangent categories
  from the coalgebras of differential categories}, in: M.~Fern{\'a}ndez and
  A.~Muscholl, editors, \emph{28th EACSL Annual Conference on Computer Science
  Logic (CSL 2020)}, volume 152 of \emph{Leibniz International Proceedings in
  Informatics (LIPIcs)}, pages 17:1--17:17, Schloss Dagstuhl--Leibniz-Zentrum
  fuer Informatik, Dagstuhl, Germany (2020), ISBN 978-3-95977-132-0, ISSN
  1868-8969.
\newline\urlprefix\url{https://doi.org/10.4230/LIPIcs.CSL.2020.17}

\bibitem{cruttwell2020categorical}
Cruttwell, G., J.~Gallagher and D.~Pronk, \emph{Categorical semantics of a
  simple differential programming language}, in: \emph{Proceedings of the 3rd
  Annual International Applied Category Theory Conference 2020, {ACT} 2020},
  volume 333 of \emph{{EPTCS}}, pages 289--310 (2020).
\newline\urlprefix\url{https://doi.org/10.4204/EPTCS.333.20}

\bibitem{cruttwell2021categorical}
Cruttwell, G., B.~Gavranovi{\'c}, N.~Ghani, P.~Wilson and F.~Zanasi,
  \emph{Categorical foundations of gradient-based learning}, in:
  \emph{Programming Languages and Systems}, pages 1--28, Springer International
  Publishing (2022).
\newline\urlprefix\url{https://doi.org/10.1007/978-3-030-99336-8_1}

\bibitem{ehrhard2003differential}
Ehrhard, T. and L.~Regnier, \emph{The differential lambda-calculus},
  Theoretical Computer Science \textbf{309}, pages 1--41 (2003).
  \newline\urlprefix\url{https://doi.org/10.1016/S0304-3975(03)00392-X}

\bibitem{fuhrmann1999direct}
F{\"u}hrmann, C., \emph{Direct models of the computational lambda-calculus},
  Electronic Notes in Theoretical Computer Science \textbf{20}, pages 245--292
  (1999).
    \newline\urlprefix\url{https://doi.org/10.1016/S1571-0661(04)80078-1}

\bibitem{garner2020cartesian}
Garner, R. and J.-S.~P. Lemay, \emph{Cartesian differential categories as skew
  enriched categories}, Applied Categorical Structures pages 1--52 (2021).
      \newline\urlprefix\url{https://doi.org/10.1007/s10485-021-09649-7}

\bibitem{ikonicoff2021cartesian}
Ikonicoff, S. and J.-S.~P. Lemay, \emph{Cartesian differential comonads and new
  models of cartesian differential categories}, Cahiers de topologie et
  g{\'e}om{\'e}trie diff{\'e}rentielle cat{\'e}goriques pages 198--239 (2023). Available online at \url{http://cahierstgdc.com/wp-content/uploads/2023/04/IKONICOFF-LEMAY-LXIV-2.pdf}

\bibitem{laird2013constructing}
Laird, J., G.~Manzonetto and G.~McCusker, \emph{Constructing differential
  categories and deconstructing categories of games}, Information and
  Computation \textbf{222}, pages 247--264 (2013).
\newline\urlprefix\url{https://doi.org/10.1016/j.ic.2012.10.015}

\bibitem{lemay2018tangent}
Lemay, J.-S.~P., \emph{A tangent category alternative to the faa di bruno
  construction}, Theory and Applications of Categories \textbf{33}, pages
  1072--1110 (2018). Available online at \url{http://www.tac.mta.ca/tac/volumes/33/35/33-35abs.html}

\bibitem{manzyuk2012tangent}
Manzyuk, O., \emph{Tangent bundles in differential lambda-categories}, arXiv
  preprint arXiv:1202.0411  (2012).
  \newline\urlprefix\url{https://doi.org/10.48550/arXiv.1202.0411}


\bibitem{moggi1991notions}
Moggi, E., \emph{Notions of computation and monads}, Information and
  computation \textbf{93}, pages 55--92 (1991).
    \newline\urlprefix\url{https://doi.org/10.1016/0890-5401(91)90052-4}

\bibitem{mulry1994lifting}
Mulry, P.~S., \emph{Lifting theorems for Kleisli categories}, in:
  \emph{Mathematical Foundations of Programming Semantics: 9th International
  Conference New Orleans, LA, USA, April 7--10, 1993 Proceedings 9}, pages
  304--319, Springer (1994).
      \newline\urlprefix\url{https://doi.org/10.1007/3-540-58027-1_15}

\bibitem{mulry2002lifting}
Mulry, P.~S., \emph{Lifting results for categories of algebras}, Theoretical
  Computer Science \textbf{278}, pages 257--269 (2002).
    \newline\urlprefix\url{https://doi.org/10.1016/S0304-3975(00)00338-8}

\bibitem{sprunger2019differentiable}
Sprunger, D. and S.~Katsumata, \emph{Differentiable causal computations via
  delayed trace}, in: \emph{2019 34th Annual ACM/IEEE Symposium on Logic in
  Computer Science (LICS)}, pages 1--12, IEEE (2019).
\newline\urlprefix\url{https://doi.org/10.1109/LICS.2019.8785670}

\bibitem{street1972formal}
Street, R., \emph{The formal theory of monads}, Journal of Pure and Applied
  Algebra \textbf{2}, pages 149--168 (1972).
  \newline\urlprefix\url{https://doi.org/10.1016/0022-4049(72)90019-9}

\bibitem{wilson2021reverse}
Wilson, P. and F.~Zanasi, \emph{Reverse derivative ascent: A categorical
  approach to learning boolean circuits}, in: \emph{Proceedings of Applied
  Category Theory 2020}, pages 247--260 (2020).
\newline\urlprefix\url{https://doi.org/10.4204/EPTCS.333.17}

\end{thebibliography}

\end{document}